\documentclass{imanum}

\usepackage[T1]{fontenc}
\usepackage[utf8]{inputenc}
\usepackage{amsmath}
\usepackage{amssymb}
\usepackage{url}
\usepackage{verbatim}
\usepackage{listings}
\usepackage{color}
\usepackage{array}
\usepackage{graphicx}
\usepackage{overpic}
\usepackage{epstopdf}
\usepackage{upquote}
\usepackage{epsfig}
\usepackage{relsize}
\usepackage{mathtools}
\usepackage{courier} 
\usepackage{xcolor} 
\usepackage{tikz} 
\usetikzlibrary{shapes,arrows}
\usepackage{soul}
\usepackage{float}
\usepackage{xcolor}

\def\urltilde{\raise.17ex\hbox{$\scriptstyle\mathtt{\sim}$}}
\newcommand{\be}{\begin{equation}}
\newcommand{\ee}{\end{equation}}

\def\C{\mathbb{C}}
\def\c{\mathbf{C}}

\def\P{\mathbf{P}}
\def\={\!\!\!=\!\!\!\!}
\newcommand\myunderline[1]{\mathbf{#1}}

\def\erf{{\rm erf}}

\newcounter{tldr}
\newenvironment{tldr}[1][]{\refstepcounter{tldr}\par\medskip
   \noindent\textbf{\ \ \ \ \ Summary:} \rmfamily}{\medskip}

   \usepackage{mathtools}

\usepackage{letltxmacro}
\LetLtxMacro\orgvdots\vdots
\LetLtxMacro\orgddots\ddots

\makeatletter
\DeclareRobustCommand\vdots{%
  \mathpalette\@vdots{}%
}
\newcommand*{\@vdots}[2]{%
  \sbox0{$#1\cdotp\cdotp\cdotp\m@th$}%
  \sbox2{$#1.\m@th$}%
  \vbox{%
    \dimen@=\wd0 %
    \advance\dimen@ -3\ht2 %
    \kern.5\dimen@
    \dimen@=\wd2 %
    \advance\dimen@ -\ht2 %
    \dimen2=\wd0 %
    \advance\dimen2 -\dimen@
    \vbox to \dimen2{%
      \offinterlineskip
      \copy2 \vfill\copy2 \vfill\copy2 %
    }%
  }%
}
\DeclareRobustCommand\ddots{%
  \mathinner{%
    \mathpalette\@ddots{}%
    \mkern\thinmuskip
  }%
}
\newcommand*{\@ddots}[2]{%
  \sbox0{$#1\cdotp\cdotp\cdotp\m@th$}%
  \sbox2{$#1.\m@th$}%
  \vbox{%
    \dimen@=\wd0 %
    \advance\dimen@ -\ht2 %
    \kern.05\dimen@
    \dimen@=.5\wd2 %
    \advance\dimen@ -\ht2 %
    \dimen2=.5\wd0 %
    \advance\dimen2 -\dimen@
    \vbox to \dimen2{%
      \offinterlineskip
      \hbox{$#1\mathpunct{.}\m@th$}%
      \vfill
      \hbox{$#1\mathpunct{\kern\wd2}\mathpunct{.}\m@th$}%
      \vfill
      \hbox{$#1\mathpunct{\kern\wd2}\mathpunct{\kern\wd2}\mathpunct{.}\m@th$}%
    }%
  }%
}
\makeatother

\title{An ultraspherical spectral method for linear Fredholm and Volterra integro-differential equations of convolution type}
\shorttitle{A US spectral method for integro-differential equations of convolution type}

\author{\sc Nicholas Hale\thanks{Department of Mathematical Sciences, Stellenbosch University, Stellenbosch, 7602, South Africa.\ (nickhale@sun.ac.za)\newline
This work is based on the research supported by the National Research Foundation of South
Africa (Grant Number 109210).}}

\begin{document}
\maketitle

\begin{abstract}{The Legendre-based ultraspherical spectral method for ordinary differential equations \citep{olver2013} is combined with a formula
for the convolution of two Legendre series \citep{hale2014b} to produce a new technique for solving 
linear Fredholm and Volterra integro-differential equations with convolution-type
kernels. When the kernel and coefficient functions  are sufficiently smooth then the
method is spectrally-accurate and the resulting almost-banded linear systems can be solved with linear complexity.}
{Fredholm, Volterra, convolution, Legendre, ultraspherical, spectral, integro-differential}
\end{abstract}%

%

\section{Introduction}\label{sec:intro}%
\noindent Fredholm and Volterra integral and integro-differential equations (IDEs) arise often in mathematical models of physical phenomena, 
such as renewal theory~\citep{cox1962}, neural networks~\citep{jackiewicz2006},  population dynamics~\citep{kuang1993}, thermodynamics~\citep{maccamy1977}, the spread of disease~\citep{medlock2003}, 
and financial mathematics~\citep{tankov2003}.
This paper concerns the numerical solution of linear Fredholm and Volterra IDEs of the 
form\be\label{eqn:fidect3}
\begin{array}{rcl}
{\cal{L}}^ry(t) &\=& f(t) + g(t)\displaystyle\int_0^T k(t-s)h(s)y(s)\,ds, \qquad 0 \leq t \leq T,
\end{array}%
\ee
and \be\label{eqn:videct3}
\begin{array}{rcl}
{\cal{L}}^ry(t) &\=& f(t) + g(t)\displaystyle\int_0^t k(t-s)h(s)y(s)\,ds, \qquad 0 \leq t \leq T,
\end{array}%
\ee
respectively, where the given functions $g(t)$ and $h(t)$ and the {\em difference} (or {\em convolution}) {kernel} $k$ are smooth,  
and ${\cal{L}}^ry(t) = a_{r}(t)y^{(r)}(t) + a_{r-1}(t)y^{(r-1)}(t) + \ldots + a_0(t)y(t)$
is a linear differential operator with smooth coefficients. 
When $r = 0$ and $g(t) = h(s) = 1$, 
then~(\ref{eqn:videct3}) reduces to the well-known {\em convolution equation}~\cite[Chapter 6]{linz1985}. 
When $r > 0$ then~(\ref{eqn:fidect3}) and~(\ref{eqn:videct3}) are known as Fredholm and Volterra integro-differential equations {\em of convolution 
type} (FIDECTs and VIDECTs), respectively, and are augmented with an auxiliary equation,  
${\cal{B}}y(t) = \gamma$, where ${\cal{B}}$ is a $r\times\infty$ linear functional 
representing initial conditions, boundary conditions, or other constraints. 

In practical applications, solutions to such equations
can not usually be obtained analytically (i.e., in closed form) and one relies on numerical approximation. 
Expanding on ideas introduced by~\citep{elgendi1969}, both~\cite{driscoll2010} and~\cite{tang2008} have recently shown that discretisation of a broad class of integral and integro-differential equations 
by collocation at Chebyshev or Legendre points can yield exponential convergence and high accuracy solutions. However, 
the drawback of global collocation is that such discretisations of the operators result in dense
matrices which are typically time consuming to invert. Furthermore, the matrices representing the differential operator on the 
left-hand side of the equations are ill-conditioned, particularly when the degree of differentiation, $r$, is large~\citep{trefethen1987}.

Olver's \lstinline{ApproxFun.jl} package~\citep{ApproxFun} uses the ultraspherical spectral (US) method \citep{olver2013} for discretisation of
differential operators, resulting in sparse (almost-banded) well-conditioned matrices and spectral convergence when solving linear ODEs. 
When the kernel $k$ in~(\ref{eqn:fidect3}) or~(\ref{eqn:videct3}) is separable, degenerate, or well-approximated 
by a low-rank function (i.e., $k(s,t)\approx\sum_{i=1}^{R}\phi_i(t)\psi_i(s)$ and $R$ is small) then
equations of the form~(\ref{eqn:fidect3})/(\ref{eqn:videct3}) can be solved
by replacing the Fredholm/Volterra operator by a linear combination of multiplication and definite/indefinite integral operators~\citep{slevinsky2016}.
Such an approach is typically referred to as {\em degenerate kernel approximation}~\citep{kress1999}.
However, when $k$ is not of low-rank (i.e., $R$ is large) this becomes computationally expensive.

The approach we take in this work is similar,
but does not rely on the kernel being separable or of low-rank. In particular, we continue to use the US 
method (in this case, with a Legendre basis) for disctretisation of ${\cal{L}}^r$, but to discretise the 
Fredholm and Volterra operators we use the formula introduced in~\citep{hale2014b} for computing the 
convolution of two Legendre series expansions. We demonstrate that when the kernel $k$ is sufficiently smooth
this results in banded linear operators which can be efficiently incorporated into  the US method framework,
obtaining spectral accuracy with linear complexity for many problems.

The outline of this paper is as follows. In Section~\ref{sec:prelim} we recall the US method and describe the algorithm used in~\citep{ApproxFun}
for solving FIDE and VIDEs when the kernel $k$ is of low-rank. In Section~\ref{sec:conveqn} we recall the formulae introduced in~\citep{hale2014b}
for the convolution of two Legendre series and show how this may be combined with US methods to solve FIDECTs and VIDECTs.
Some examples and numerical results are given in Section~\ref{sec:examples} before we conclude in Section~\ref{sec:conc}.

{\bf Remark:} It is worthwhile to note that the method described in this paper is \underline{not} directly applicable 
to fractional integral and differential equations (FIEs and FDEs). Although FIEs take the form of~(\ref{eqn:videct3}), 
the kernel in such cases is not smooth (it has end-point singularities) and hence its Legendre series expansion will
converge too slowly. For an ultraspherical method for  FIEs and FDEs, see~\citep{hale2017}.

{\bf Remark:} We use the following conventions in our notation throughout: Calligraphic fonts (e.g., ${\cal{L}}$
and ${\cal{D}}$) represent continuous linear operators acting on functions. 
Uppercase characters (e.g., $L$ and $D$) represent
(possibly infinite-dimensional) matrices. 
Bold face uppercase math fonts which take an argument (e.g., $\mathbf{P}(x)$) represent {\em quasimatrices}, i.e., 
matrices whose columns are continuous functions.
Underlined variables represent 
(possibly infinite-dimensional) vectors containing coefficients of mapped-Legendre 
series expansions of a function on the interval $[0,T]$.
Indicies for vectors and matrices begin at zero. MATLAB to reproduce all of the numerical results in this 
paper can be found in the Git repository~\citep{hale2017_code}.\\

{\bf Summary:} {We solve convolution-type integro-differential equations by expanding the solution and the kernel in a Legendre basis. The resulting matrices are sparse and the convergence is fast.}%
\section{Preliminaries}\label{sec:prelim}%
\subsection{Legendre polynomials and Legendre series}\label{subsec:legpolys}
The Legendre polynomials $\widehat P_0, \widehat P_1, \ldots$ are orthogonal with respect to the $L^2$ inner-product on the interval $[-1,1]$, 
with the conventional normalisation $||\widehat P_n||^2_2 = 2/(2n+1)$.\footnote{We include the $\ \widehat{} \ $ notation here to distinguish from the
mapped-Legendre polynomials introduced momentarily.}
We denote by $\widehat\P(x)$ the Legendre quasimatrix, whose $j$th column is the degree $j-1$ Legendre polynomial, $\widehat P_{j-1}(x)$, i.e.,
$
\widehat\P(x) = [\widehat P_0(x), \widehat P_1(x), \ldots].
$
Since~(\ref{eqn:fidect3}) and~(\ref{eqn:videct3}) are defined not on the interval $[-1,1]$ but rather on $[0,T]$, 
we introduce the mapped-Legendre polynomials and their associated quasimatrix
\be
P_n(t) = (\widehat P_n\circ\psi_{[0,T]})(t), \qquad \P(t) = [P_0(t), P_1(t), \ldots],
\ee
where $\psi_{[0,T]}(t) = 2t/T-1$ is the linear map from $[0,T]$ to $[-1,1]$. The columns of $\P(t)$ then form a basis for Lipschitz continuous 
functions on $[0, T]$. We associate with the basis $\P(t)$ a space of 
coefficients, $\P\cong\C^\infty$, so that $\underline{v} = (v_0 , v_1 , . . .)^\top \in\P$ if and only if
\be
\sum_{n=0}^\infty|v_n|\sup_{0\le t\le T}|\widehat P_n(t)| =\sum_{n=0}^\infty|v_n|\sup_{-1\le x\le1}|P_n(x)| = \sum_{k=0}^\infty|v_n| < \infty,
\ee
thus if $\underline{v}\in\P$ then $\P(t)\underline{v}$ defines a continuous function on $[0, T]$. If all but a finite number of the $v_n$ are non-zero, 
then $p_N(t) = \P(t)\underline{v}$ is a polynomial of degree $N$, where $v_N$ is the final non-zero coefficient in $\underline{v}$.

Conversely, it follows directly from the orthogonality of the Legendre polynomials that given a continuous function $p(x)$ on $[-1,1]$, 
the best-$L^2$ degree $N$ polynomial approximation of $p$ is obtained by truncation of its Legendre series.
Formally, the $k$th coefficient in the expansion of $p$
is given by
\be\label{eqn:legcoeffs}
v_n = \frac{2n+1}{2}\int_{-1}^{1}\widehat P_n(x)f(x)\,dx,
\ee
but obtaining the coefficients in this manner is computationally expensive, particularly if $N$ is large (see discussion below).
The connection to best-$L^2$ approximation suggests that the approximation of a smooth function by a series 
of Legendre polynomials will converge quickly as the degree of the expansion increases.
For analytic functions, exactly how fast the series converges is made precise by the following result:
\begin{theorem}[\cite{wang2012}]\label{thm:leg}
Denote by $v_N(x)$ the truncated Legendre series expansion of $v(x)$, i.e, 
$
v_N(x) = \sum_{n=0}^{N}v_nP_n(x).
$
If $v(x)$ is analytic inside and on a Bernstein ellipse ${\cal{E}}_\rho$ then for each $N\ge 0$,
\be\max_{x\in[-1,1]}|v(x)-v_N(x)| \le \frac{(2N\rho+3\rho-2N-1)\ell({\cal{E}}_\rho) M}{\pi\rho^{N+1}(\rho-1)^2(1-\rho^{-2})} = {\cal{O}}(\rho^{-N-3}),
\ee
where ${\cal{E}}_\rho = \left\{z\in\mathbb{C} \ : \ z = \frac12(\rho e^{i\theta} + \rho^{-1}e^{-i\theta}), \ \rho \ge 1, \ 0 \le \theta \le 2\pi\right\}$ and $M=\max_{z\in{\cal{E}}_\rho}|v(z)|.$
\end{theorem}

Note in particular that this convergence is geometric as $N$ is increased. 
\cite{wang2012} provide a similar result for when $v(x)$ is smooth, but not analytic.
In such cases convergence is algebraic in $N$, with the degree roughly the number of continuous derivatives of $v(x)$ on $[-1,1]$.
However, as mentioned above, computing Legendre coefficients from the definition~(\ref{eqn:legcoeffs}) is usually impractical.
Instead, a more efficient paradigm is to first compute approximate Chebyshev coefficients of 
$p$ by interpolating on a Chebyshev grid and applying a discrete cosine transform using the FFT~\cite[Section 4.7]{mason2003}. From these, the corresponding 
approximate Legendre coefficients can be rapidly computed by any of the algorithms presented in~\citep{alpert1991, hale2014a, townsend2016}.
In such a situation it is rather the approximation properties of the Chebyshev interpolant which are required (see, for example, \cite[Chapters 7 \& 8]{trefethen2013}),
but since (a) these are asymptotically very similar to that of the truncated Legendre series and (b) a detailed analysis of 
the convergence of the presented  algorithm is beyond the scope of this paper, we skip the details.


\tldr{Legendre series converge exponentially fast for analytic functions. There are efficient algorithms for approximating the coefficients.}

\subsection{Ultraspherical polynomials and banded operators}\label{sec:uspoly}

The Legendre polynomials are a particular case of the more general ultraspherical (or {\em Gegenbauer}) polynomials, $\widehat C_n^{(\lambda)}(x)$, $\lambda > 0$,
orthogonal on the interval $[-1,1]$ with respect to the weight function $(1-x^2)^{\lambda-1/2}$. As above, one can define a quasimatrix
of mapped ultraspherical polynomials, ${\mathbf{C}}^{(\lambda)}(t)$,  on the interval $[0, T]$ via an affine transformation and associate
with them a space of coefficients ${\mathbf{C}}^{(\lambda)}$.

If $\underline{v}\in \c^{(\lambda)}$ then linear operators which can be applied to $v(t) = \c^{(\lambda)}(t)\underline{v}$ induce 
infinite-dimensional matrices that can be viewed as acting on $\underline{v}$.    
For example, given a continuous linear operator ${\cal{L}}$ satisfying
\be
    {\cal L}C^{(\lambda)}_n(t)= \sum_{j=n-m}^{n+m} L_{jn} C^{(\ell)}_n(t),
\ee
we  can associate it with an $m$-{\em banded} (i.e., banded  with bandwidth $m$) infinite-dimensional matrix
\be\label{eqn:L}
	L := \begin{bsmallmatrix} L_{00} & \cdots & L_{0m} \cr 
			\vdots  & \ddots & L_{1m} & L_{1,m+1}\cr 
			L_{m0} & L_{m1} &  \ddots & L_{mm} & \ddots \cr
					  & L_{m+1,1} & \ddots & \ddots& \ddots \cr
					  && \ddots & \ddots& \ddots
	\end{bsmallmatrix}.
\ee
Since $L$ is banded, multiplication is well-defined on $\C^\infty$ and~(\ref{eqn:L}) can be viewed as an operator $L: \c^{(\lambda)}\rightarrow\c^{(\ell)}$. 
To relate $L$ and ${\cal L}$ we note that (by construction) we have
\be
	 {\cal L}C^{(\lambda)}_n(t)= {\cal L} \c^{(\lambda)}(t) \underline{e}_n= \c^{(\ell)}(t) L \underline{e}_n.
\ee
If $v(t)$ is Lipschtiz continuous on $[0,T]$ then there exists $\myunderline{v}\in\c^{(\lambda)}$ so that $\c^{(\lambda)}(t) \underline{v} = v(t)$ and because ${\cal L}$ is continuous, we have 
\be
 	{\cal L} v = {\cal L} \c^{(\lambda)}(t)\underline{v} = \c^{(\ell)}(t) L \underline{v}.
\ee
Therefore, applying ${\cal{L}}$ to $v(t)$ is equivalent to applying ${L}$ to $\underline{v}$.

Such banded operators form the basis of the ultraspherical spectral method, which we describe below. 
In Section~\ref{sec:conveqn} we will show that Fredholm and Volterra operators also give rise 
to such banded operators when the kernel is sufficiently smooth.

\tldr{Linear operators can be thought of as infinite dimensional matrices acting on spaces of Legendre and ultraspherical polynomial coefficients.}

\subsection{The ultraspherical spectral method}\label{subsec:us} 
Traditional Chebyshev--Galerkin spectral methods  (as introduced by~\cite{lanczos1956} and later developed by~\cite{ortiz1969} and~\cite{orszag1971b,orszag1971a}) 
expand both the unknown solution and the right-hand side of an ordinary differential equation (ODE) as
Chebyshev series of the first kind. One then applies the differential operator 
to the Chebyshev series and equates coefficients (or minimises the residual in some
suitable norm) to obtain the approximate solution. Although resulting in rapidly converging 
approximations (often geometric in the number of degrees of freedom) such an approach typically leads to dense 
and ill-conditioned matrices~\citep{trefethen1987}.

The key observation of the ultraspherical spectral (US) method introduced by~\cite{olver2013} is that if one is willing to map between Chebyshev and higher-order ultraspherical bases (i.e., a Petrov--Galerkin approach)
then differentiation becomes a banded operator. 
After a diagonal preconditioner, the discretization linear systems have a condition number that is bounded by a constant that is independent of the dsicretisation size. This allows a broad
class of ODEs (including some variable coefficient problems) to be solved with spectral 
accuracy and linear complexity, and overcomes the ill-conditioning associated with spectral differentiation matrices. 
Although the US method as described in~\citep{olver2013}
expands the solution as a Chebyshev series of the first kind (which is usually the 
most convenient to work with), 
this is not necessary. Indeed, \cite{hale2017} 
recently used Legendre expansions as part of an ultraspherical 
method for fractional differential equations. 

The Legendre-US method for IDEs requires four banded operators which act on Legendre and related ultraspherical polynomials: differentiation, integration, conversion, and multiplication. 
Derivation of these operators on the interval $[-1,1]$ can be found in~\cite[Section 2]{hale2017}, but for convenience, we give examples of the 
first three acting on $\P$ (i.e., on the interval $[0, T]$) in Table~\ref{tab:ops}. {Multiplication operators, $M_\lambda[f]$, are less easy to write down (since they depend on the coefficients
of the multiplying function), but they are readily generated using the procedure described in~\cite[Section 2.3]{hale2017} and we omit the details here.}
MATLAB code for computing these matrices for a given value of $N$ can be found~\citep{hale2017_code}.

\begin{table}\begin{center}        
{\small \bgroup
\def\arraystretch{1.2}
\setlength\tabcolsep{-.15mm}
\caption{\noindent The banded infinite-dimensional matrices representing differentiation, integration, and conversion operators
when acting on the spaces of mapped-Legendre and ultraspherical coefficients on the interval $[0, T]$. 
One should interpret this table as follows: If $v(t)\in \P(t)$ so that $v(u) = \P\underline{v}$ then, for example, $\frac{d}{dt}v(t) = {\cal{D}}\P(t)\underline{v} = \c^{(3/2)}(t){D_{1/2}}\underline{v}$
and $v(t) = \P(t)\underline{v} = \c^{(3/2)}(t)S_{1/2}\underline{v}$.
These relationships may be used to solve linear differential, 
integral, or integro-differential equations as outlined below (see, for example, equation~(\ref{eqn:usexamplelhs})).
For derivations see~\citep{hale2017} or~\citep{olver2013}.
MATLAB code for computing these matrices can be found in~\citep{hale2017_code}.}
\begin{tabular}{|ccrcc | ccrcc|}\hline
$\begin{array}{c}\text{\tiny Continuous}\\[-.9em] \text{\tiny operator}\end{array}$ & $\begin{array}{c}\text{\tiny Discrete}\\[-.9em]  \text{\tiny Operator}\end{array}$ & $\begin{array}{c}\text{\tiny Matrix}\\[-.9em] \text{\tiny Representation}\end{array}\hspace*{20pt}\mbox{}$ & \tiny Range & \tiny $\,\,\,\,\mbox{}$   Image $\,\,\,\,\mbox{}$ &
$\begin{array}{c}\text{\tiny Continuous}\\[-.9em]  \text{\tiny operator}\end{array}$ & $\begin{array}{c}\text{\tiny Discrete}\\[-.9em]  \text{\tiny Operator}\end{array}$ & $\begin{array}{c}\text{\tiny Matrix}\\[-.9em] \text{\tiny Representation}\end{array}\hspace*{20pt}\mbox{}$ & \tiny Range & \tiny $\,\,\,\,\mbox{}$   Image $\,\,\,\,\mbox{}$ \\\hline&&&&&&&&&\\[-.9em]
  ${\cal{D}} = \frac{d}{dt}$ & $D_{1/2}$ & $\tfrac{2}{T}\begin{bsmallmatrix}0&1&\\&0&1\\&&0&1\\\phantom{-\ddots}&\phantom{-\ddots}&\phantom{-\ddots}&\phantom{-}\ddots&\phantom{-}\ddots\end{bsmallmatrix}\!\ $ & $\P$ & $\c^{(3/2)}$ &
  ${\cal{D}}^2 = \frac{d^2}{dt^2}$ & $D^2_{1/2}$ & $\tfrac{12}{T^2}\begin{bsmallmatrix}0&0&1&\\&0&0&1\\&&0&0&\phantom{-}\ddots\\\phantom{-\ddots}&\phantom{-\ddots}&\phantom{-\ddots}&\phantom{-}\ddots&\phantom{-}\ddots\end{bsmallmatrix}$ & $\P$ & $\c^{(5/2)}$ \\[.55em]
  
  $\int_0^t\cdot\,dt$ & $Q_{1/2}$ & $\tfrac{T}{2}\begin{bsmallmatrix}1&-\frac{1}{3}\\1&0&-\frac{1}{5}\\&\frac{1}{3}&0&-\frac{1}{7}\\\phantom{-\ddots}&\phantom{-\ddots}&\phantom{-}\ddots&\phantom{-}\ddots&\phantom{-}\ddots\end{bsmallmatrix}$ & $\P$ & $\P$ &
  $\int_0^T\cdot\,dt$ & $Q^{\rm def}_{1/2}$ & $T\begin{bsmallmatrix}1&&\\&0\\&&0\\\phantom{\ddots}&\phantom{\ddots}&\phantom{\ddots}&0&\phantom{\ddots}\\\phantom{-\ddots}&\phantom{-\ddots}&\phantom{-\ddots}&\phantom{-\ddots}&\phantom{-}\ddots\end{bsmallmatrix}$ & $\P$ & $\P$\\
  
  ${\cal{I}}$ & $S_{1/2}$ & $\begin{bsmallmatrix}1&0&-\frac{1}{5}\\&\frac{1}{3}&0&-\frac{1}{7}\\&&\frac{1}{5}&0&\phantom{-}\ddots\\\phantom{-\ddots}&\phantom{-\ddots}&\phantom{-\ddots}&\phantom{-}\ddots&\phantom{-}\ddots\end{bsmallmatrix}$ & $\P$ & $\c^{(3/2)}$ &
  ${\cal{I}}$ & $S_{3/2}$ & $3\begin{bsmallmatrix}\frac{1}{3}&0&-\frac{1}{7}\\&\frac{1}{5}&0&-\frac{1}{9}\\&&\frac{1}{7}&0&\phantom{-}\ddots\\\phantom{-\ddots}&\phantom{-\ddots}&\phantom{-\ddots}&\phantom{-}\ddots&\phantom{-}\ddots\end{bsmallmatrix}$ & $\c^{(3/2)}$ & $\c^{(5/2)}$\\\hline
\end{tabular}\egroup
}
\label{tab:ops}
\end{center}\end{table}

Using operators from Table~\ref{tab:ops}, we demonstrate the Legendre-based US method on the first-order IDE
\be\label{eqn:usexample}
y'(t) + ay(t) = e^{t}\int_{0}^te^{-s}y(s)\,ds, \quad 0 \le t \le 1, \qquad y(0) = 1.
\ee
To proceed, we assume that the solution to~(\ref{eqn:usexample}) is sufficiently smooth so that we may consider its 
mapped-Legendre series expansion $y(t) = \P(t)\underline{y}$. Discretisation of the left-hand side of~(\ref{eqn:usexample})
proceeds almost exactly as described in~\cite[Section 2]{olver2013}, but with ${\cal{D}}_0$ and ${\cal{S}}_0$
replaced by $D_{1/2}$ and $S_{1/2}$ from Table~\ref{tab:ops}, respectively, so that 
\be\label{eqn:usexamplelhs}
({\cal{L}}y)(t) := y'(t) + ay(t) = \c^{(3/2)}(t)\underbrace{(D_{1/2} + aS_{1/2})}_{L}\underline{y}.
\ee
Notice in particular that the result is now expressed in the ultraspherical basis, $\c^{(3/2)}$, and that $L$ is banded (since it is a linear combination of two banded matrices).

For the right-hand side we require multiplication operators $M_{1/2}[\cdot]$ and the indefinite integral operator $Q_{1/2}$ so that 
\be\label{eqn:usexamplerhs}
e^{t}\int_{0}^te^{-s}y(s)\,ds = \P(t) \underbrace{M_{1/2}[\underline{e}^{t}]Q_{1/2}M_{1/2}[\underline{e}^{-s}]}_{V}\underline{y} = \c^{(3/2)}(t) S_{1/2}V\underline{y},
\ee
where in the second equality we have multiplied by $S_{1/2}$ to make the bases of~(\ref{eqn:usexamplelhs}) and~(\ref{eqn:usexamplerhs}) compatible.
Therefore~(\ref{eqn:usexample}) is equivalent to
\be\label{eqn:usexample2}
\underbrace{(L - S_{1/2}V)}_A\underline{y} = 0.
\ee
If the coefficients of $e^{t}$ and $e^{-s}$ are truncated at machine precision, then each of the 
operators that make up $A$ are banded, and hence $A$ is banded (with bandwidth around 15 in this example).
The initial condition  $y(0) = 1$ can be expressed as $B\underline{y} := \P(0)\underline{y} = [1, -1, 1, -1, 1, -1,\ldots]\underline{y} = 1$, and hence 
we arrive at the set of linear equations 
\be
      \begin{bmatrix}\ B \ \mbox{} \\ \ L - S_{1/2}V \ \mbox{}
      \end{bmatrix}
      \underline{y} = 
      \begin{bmatrix} 1\\\underline{0}\end{bmatrix}.
\ee
The infinite dimensional matrix on the left-hand side of the above is now {\em almost-banded}, that is, banded with one dense row (see the left panel of Figure~\ref{fig:example5}).
One can solve for the approximate coefficients of the solution by truncating the matrix equation at a suitable size $N$. Using 
a Schur-complement factorization of the 1-1 block, or the Woodbury formula, this can be done in ${\cal{O}}(N)$ operations.
Alternatively, one can employ the adaptive-QR method, which does not truncate but rather `solves the infinite dimensional problem'
until the coefficient of the solution fall below a specified tolerance~\citep{olver2013,olver2014}.
In either case, since the kernel is smooth (in this case entire) fast convergence to the underlying solution is usually observed.

\tldr{The unknown solution is expanded as a Legendre series and the right hand-side as higher-order ultraspherical polynomials.
This leads to a well-conditioned and almost-banded linear system, which can be efficiently solved.}

\subsection{FIDECTS and VIDECTS with low-rank kernels}\label{sec:usfidects}
A bivariate function $k(s,t)$ is {\em separable} (or  {\em degenerate}) if it can be expressed as
\be\label{eqn:rank}
k(s,t) = \sum_{i=1}^{R}\phi_i(t)\psi_i(s), \qquad R < \infty.
\ee
If~(\ref{eqn:rank}) holds then we say that $k$ has {\em rank} at most $R$, 
and if $R$ is small we say that $k$ is of {\em low rank}.
Even if $k$ is not of finite rank, it can often be well-approximated (in some appropriate norm) 
by a finite-rank function so that the equality~(\ref{eqn:rank}) can be replaced by $\approx$.
Adaptive low rank approximations of bivariate functions can be efficiently computed by Chebfun2 in Chebfun~\citep{chebfun2, Chebfun}.

Substituting~(\ref{eqn:rank}) to a Volterra operator yields
\be\label{eqn:lowrankvolt}
\int_{0}^tk(s,t)y(s)\, ds = \sum_{i=1}^{R}\phi_i(t)\int_{0}^t\psi_i(s)y(s)\,ds,
\ee
hence the Volterra operator is equivalent to a linear combination of multiplication and indefinite integral operators~\cite[Chapter 11]{kress1999}.
To discretise the operator using the Legendre-US method, we proceed as above, i.e., expand $y$, the $\phi_o$, and the $\psi_o$ as Legendre series so that~(\ref{eqn:lowrankvolt}) becomes
\be\label{eqn:lowrankvolt2}
\int_{0}^tk(s,t)y(s)\, ds =  \widehat\P(t)\underbrace{\left[\sum_{i=1}^{R}M[\underline{\phi}_i]Q_{1/2}M[\underline{\psi}_i]\right]}_{V}\underline{y}.
\ee
To solve VIDEs of the form 
\be\label{eqn:volt}
({\cal{L}}^{r}y)(t) = f(t) + \int_{0}^tk(s,t)y(s)\, ds,
\ee
subject to ${\cal{B}}y(t) = {\gamma}$,
one forms a Legendre-US discretisations $L$ of the differential operator ${\cal{L}}^r$ and 
$B$ of the functional ${\cal{B}}$,  and then solves
\be
\begin{bmatrix}
 B\\ L - S_{[r]}V
\end{bmatrix}
      \underline{y} = 
\begin{bmatrix}
  {\gamma}\\S_{[r]}\underline{f}
\end{bmatrix}
\ee
where $S_{[r]} = S_{r+1/2}S_{r-1/2}\ldots S_{3/2}S_{1/2}:\P\rightarrow\c^{r+1/2}$. In fact, the VIDE solved in Section~\ref{subsec:us}
is precisely the Volterra equation~(\ref{eqn:volt}) with $({\cal{L}}^ry)(t) = y'(t) + ay(t)$, $f(t) = 0$, the rank-1 kernel $k(s,t) = \exp(-(t-s))$,
and initial condition ${\cal{B}}y = y(0) = 1$. 

Fredholm operators, where the upper limit of integration in~(\ref{eqn:lowrankvolt2})
becomes $T$ rather than $t$, can be represented similarly, with the matrix $Q_{1/2}$ in~(\ref{eqn:lowrankvolt})
replaced by $Q_{1/2}^{\rm def}$ from Table~\ref{tab:ops}. 
The approach can also be extended to 
equations involving combinations of multiple Fredholm and Volterra operators, 
systems of linear integro-differential equations (where the almost-banded 
structure of the operators can be maintained by interlacing the coefficients of the variables),
and partial integro-differential equations.

\tldr{If the kernel is degenerate or well-approximated by a low rank function, then Fredholm and Volterra operators
are simplified to definite and indefinite integrals.}
\section{New method}\label{sec:conveqn}%
One drawback with the algorithm described above is that it requires a separable decomposition of $k(s,t)$. 
Furthermore, if $R$ is large, then constructing the operator $V$ in this manner may be time consuming. 
In this section we propose a modification
to the approach which avoids the need to compute low-rank
approximations when the kernel is of convolution type, i.e., $k(s,t) = k(t-s)$.

\subsection{A banded convolution operator}\label{subsec:convop}
Consider $\underline{k},\underline{y}\in\P$ so that $k(x) = \P(x)\underline{k}$ and $y(x) = \P(x)\underline{y}$ are Lipschitz continuous functions on $[0, T]$. 

\begin{proposition}\label{prob:banded}
Suppose $\underline{k}$ has a finite number of nonzero entries so that $k(t) = \P(t)\underline{k}$ is a polynomial of degree $m$, 
then the convolution operator 
\be
({\cal{V}}[k]y)(t) := \int_{0}^tk(t-s)y(s)\,ds, \qquad 0 \le t \le T, 
\ee
corresponds to an $(m+2)$-banded infinite dimensional matrix $V[\underline{k}]:\P\rightarrow\P, \text{ i.e.},$
\be ({\cal{V}}[k]y)(t) = \P(t)V[\underline{k}]\underline{y}.\ee
We refer to $V[\underline{k}]$ as the {\em discrete Legendre convolution operator}.
\end{proposition}%
\begin{proof}
Since $k(t)$ is Lipschitz continuous on $[0, T]$ 
then ${\cal{V}}[k]$ is a compact operator and therefore there exists $\underline{u}\in\P$ such that $u(t) = \P(t)\underline{u} = ({\cal{V}}[k]y)(t)$.
Modifying the derivation in~\cite[Section 4]{hale2014b} to account for the fact that we consider the interval $[0, T]$, we have from~(\ref{eqn:legcoeffs})
that for $j = 0, 1, \ldots$ the mapped-Legendre coefficients of $u(t)$ satisfy
\begin{eqnarray}
 u_j 
  &=& \frac{2j+1}{T}\int_{0}^T P_j(t)\int_0^tk(t-s)y(s)\,ds\,dt,\\
     &=& \sum_{n=0}^\infty \underbrace{\left[\frac{2j+1}{T}\int_{0}^T P_j(t)\int_0^tk(t-s) P_n(s)\,ds\,dt\right]}_{V[\underline k]_{j,n}}  y_n = [V[k]\underline{y}]_j
\end{eqnarray}
where in the second equality we have used the absolute convergence of the coefficients $\underline{y}$.
Hence we have that $\underline{u} = V[\underline{k}]\underline{y}$. It remains to show that $V[\underline{k}]$ is banded. 

It was shown in~\cite[Theorem 4.1]{hale2014b} that the entries of $V[\underline k]$ satisfy the recurrence relation
\be\label{eqn:V_recurrence}
V[\underline{k}]_{j,n+1} = -\frac{2n+1}{2j+3}V[\underline{k}]_{j+1,n} + \frac{2n+1}{2j-1}V[\underline{k}]_{j-1,n} + V[\underline{k}]_{j,n-1}, \quad n \ge 1,
\ee
with starting values (again, modified to account for the fact we are on the interval $[0,T]$)
\begin{eqnarray}\label{eqn:V_recurrence2}
V[\underline{k}]_{j,0} &=& \left\{\begin{array}{ll}{\tfrac{T}{2}[k_{j-1}}/({2j-1})-{k_{j+1}}/({2j+3})], \hspace*{73pt} \mbox{}& j\not=0,\\[.1em]
\tfrac{T}{2}[k_0-k_1/3],&j = 0,\end{array}\right.\label{eqn:Bk0}\\
V[\underline{k}]_{j,1} &=& \left\{\begin{array}{ll}V[\underline{k}]_{j-1,0}/(2j-1)-V[\underline{k}]_{j,0}-V[\underline{k}]_{j+1,0}/(2j+3),\hspace*{6pt} &  j\not=0,\\[.1em]
-V[\underline{k}]_{1,0}/3,&j = 0,\end{array}\right.\label{eqn:Bk1}\\
V[\underline{k}]_{0,n} &=& \left.\begin{array}{ll}(-1)^nV[\underline{k}]_{n,0}/(2n+1), & \hspace*{127pt} n \ge 0.\end{array}\right.\label{eqn:V_recurrence3}
\end{eqnarray}
Since $k(t)$ is a polynomial of degree $m$, we have $k_j = 0$ for $j > m$ and hence from~(\ref{eqn:Bk0}) that  $V[\underline{k}]_{j,0} = 0$ for $j > m + 1$. Subsequently, from~(\ref{eqn:Bk1}),
we also have that $V[\underline{k}]_{j,1} = 0$ for $j > m + 2$, 
and from the recurrence~(\ref{eqn:V_recurrence}) it follows by induction that $V[\underline{k}]$ is lower-banded with bandwidth $m+2$.
The fact that $V[\underline{k}]$ is also upper-banded, and hence banded, follows from the scaled-symmetry relation~\cite[(4.8)]{hale2014b}:
\be
V[\underline{k}]_{j,n} = (-1)^{n+j}\frac{2j+1}{2n+1}V[\underline{k}]_{n,j} \qquad j, n \ge 0.
\ee
\end{proof}\newpage

\tldr{The convolution operator acting on the space of Legendre coefficients satisfies the recurrence relation~(\ref{eqn:V_recurrence})--(\ref{eqn:V_recurrence3}).
For polynomial kernels this operator is banded (with bandwidth roughly the degree of the polynomial). 
Sufficiently smooth but non-polynomial kernels can be approximated arbitrarily well by a polynomial.}

\subsection{Solving Volterra IDEs}

We are now in a position to solve Volterra equations of the form~(\ref{eqn:videct3}).
Let $f(t)$, $g(t)$, $h(t)$, and $k(t)$ be Lipschitz continuous functions on $[0, T]$. Rearranging~(\ref{eqn:videct3}) and expressing the result in operator notation 
so that ${\cal{A}} = ({\cal{L}}^r - g{\cal{V}}[k]h)$, then we have
\be\label{eqn:conveqnopform}
({\cal{A}}y)(t) = f(t), \qquad 0\le t \le T,
\ee
which is equivalent to
\be\label{eqn:conveqndisc}
 A\underline{y} := \left(L - S_{[r]}M[\underline{g}]V[\underline{k}]M[\underline{h}]\right)\underline{y} = S_{[r]}\underline{f},
\ee
where $L$ is the Legendre-US discretisation of ${\cal{L}}^r$ and $V[\underline{k}]$ is as in Section~\ref{subsec:convop}.

\begin{example}\label{example:5}
Consider the following example from~\citep{ma2006},
 \be\label{eqn:ex5}
 y'(t) + ay(t) = \int_0^te^{-(t-s)}y(s)\,ds, \quad 0\le t \le 1, \quad y(0) = 1,
 \ee
 where $a$ is a given constant. The exact solution is given by 
 \be\label{eqn:ex5sol}
 y(t) = e^{-\frac{a+1}{2}t}\left(\cosh(bt) + \tfrac{1-a}{2b}\sinh(bt)\right),
 \ee
 where $b = \tfrac{1}{2}\sqrt{a^2-2a+5}$.
 Here $g(t) = h(s) = 1$, so we need only compute the Legendre coefficients of $e^{-t}$ on $[0, T]$ in order to construct the 
 Volterra operator as described above. These can be computed in quasi-linear time using recently developed fast algorithms~\citep{alpert1991, hale2014a, townsend2016}, 
 or in linear time by observing that $k(t)$ satisfies $k'-k = 0$, $k(0) = 1$ and solving this simple ODE on $[0, 1]$ using
 the US-Legendre method. In either case, the resulting discretisation is given by
 \be\label{eqn:ex5spy}
 \begin{bmatrix} 
 \begin{bmatrix}
 1,&-1,&1,&-1,&\ldots
 \end{bmatrix}\\
 D_{1/2} + aS_{1/2}(I - V[\underline{e}^{-t}])\end{bmatrix}\underline{y} = \begin{bmatrix}1\\\underline{0}\vphantom{S_{1/2}}\end{bmatrix}.
 \ee
 
 Figure~\ref{fig:example5} shows the discretisation, solution, and convergence of the Legendre US method for this VIDECT when $a = 100.$ 
 The first panel verifies that the operator in~(\ref{eqn:ex5spy}) is almost-banded, 
 with one dense row at the top due to the boundary condition $y(0)=1$, and the bandwidth of the remainder determined by the decay in the Legendre coefficients of the kernel.
 It takes around 13 terms in a Legendre series expansion to represent $e^{-t}$ to machine precision on $[0, 1]$, so here the bandwidth of the operator is around 15 (13+2).
 This drops to around 10 if only 8 digits of precision are required. 
 The centre panel shows a plot of the obtained solution. The approximated Legendre series returned upon solving~(\ref{eqn:conveqndisc})
 is evaluated using Clenshaw's algorithm. Since the solution here has such an rapid initial decay, we show also the logarithm (base-10) of the solution.
 The third panel shows the error (approx infinity norm error at 1000 equally-spaced points, again evaluated using Clenshaw's algorithm) in the computed solution 
 as $N$ is increased. Because the solution $y(t)$
 is entire we observe super-geometric convergence until this plateaus at around $10^{-15}$ 
 due to rounding error in evaluating the computed Legendre series. The well-conditioned discretisation of the differentiation operator 
 means there is no growth in the error even for very large values of $N$.%
\begin{figure}[t]
 \includegraphics[height=105pt, trim={0 15pt 0 0},clip]{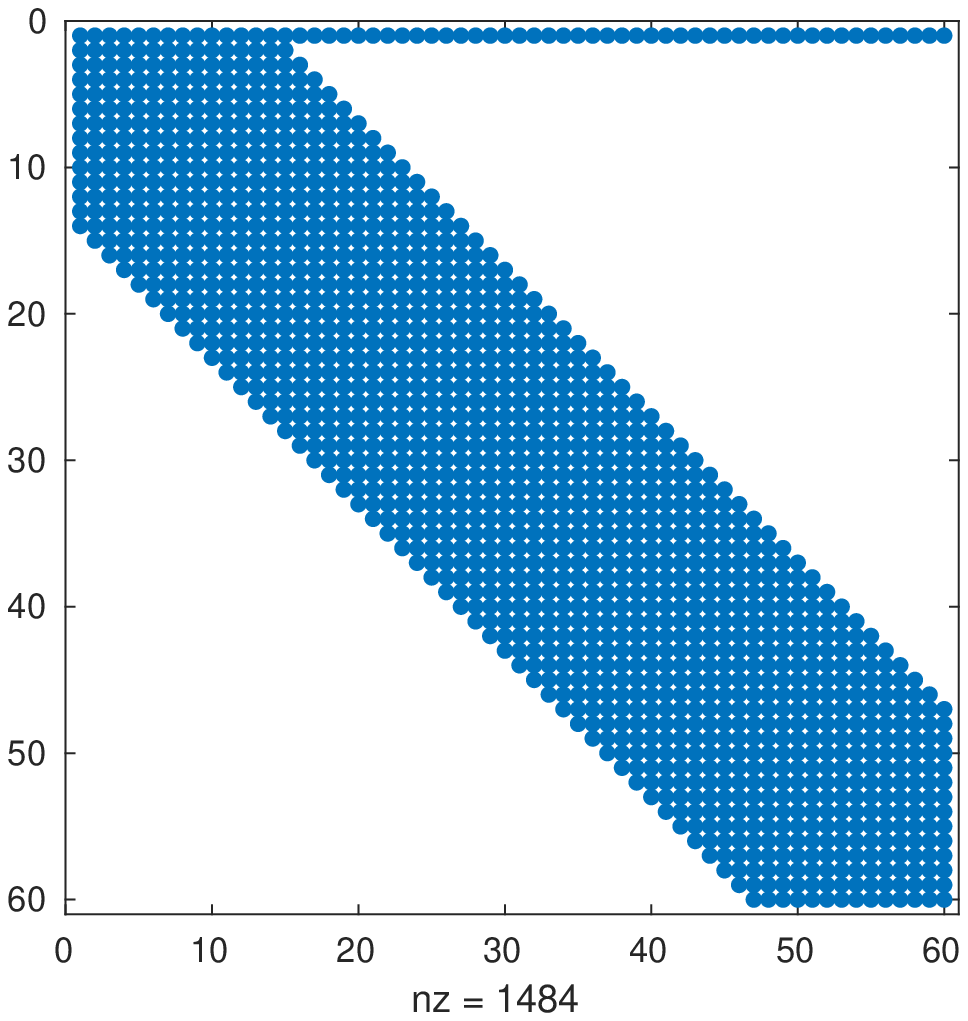}\hfill\hspace*{5pt}
 \includegraphics[height=105pt]{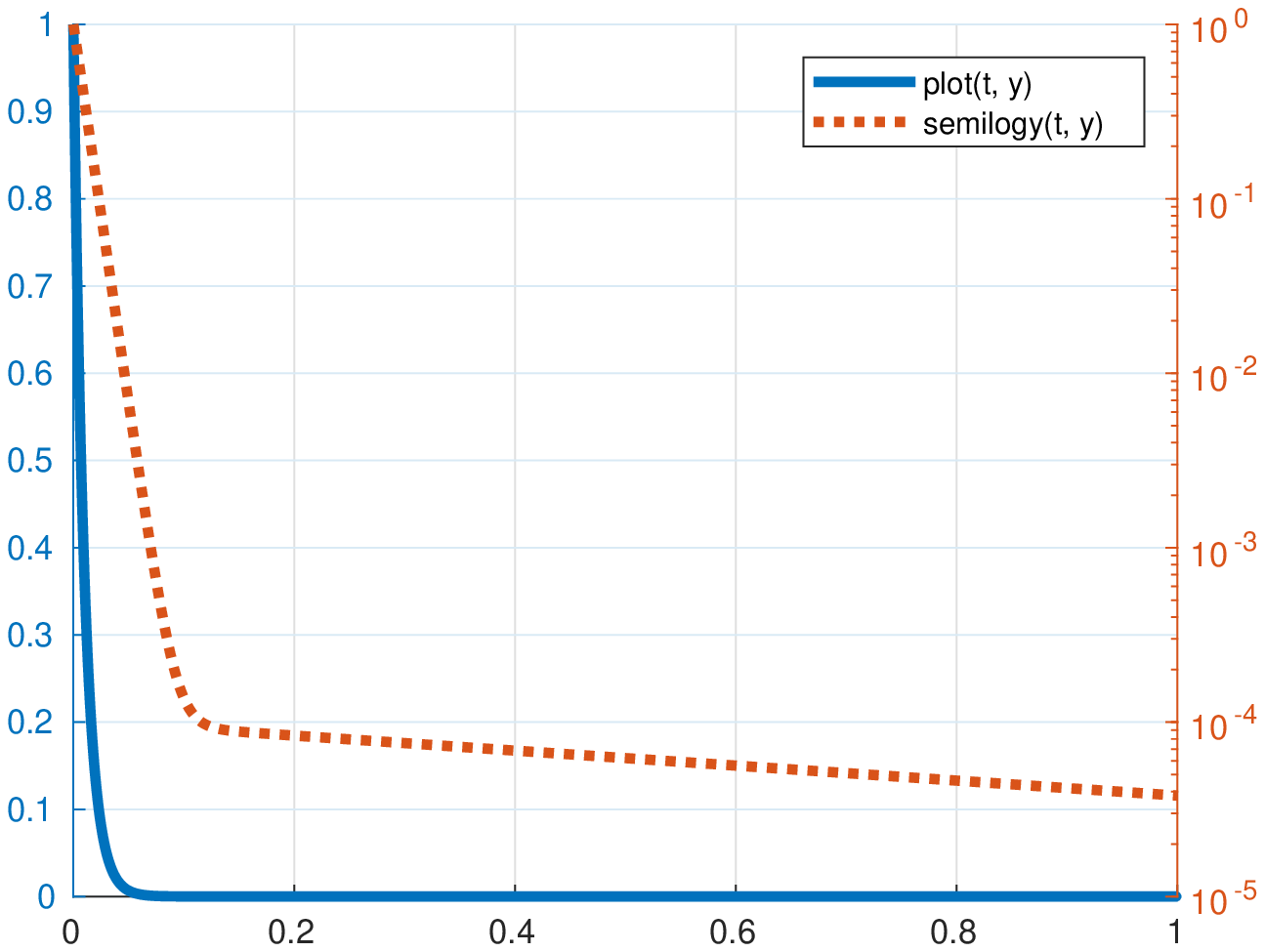}\hfill
 \includegraphics[height=105pt]{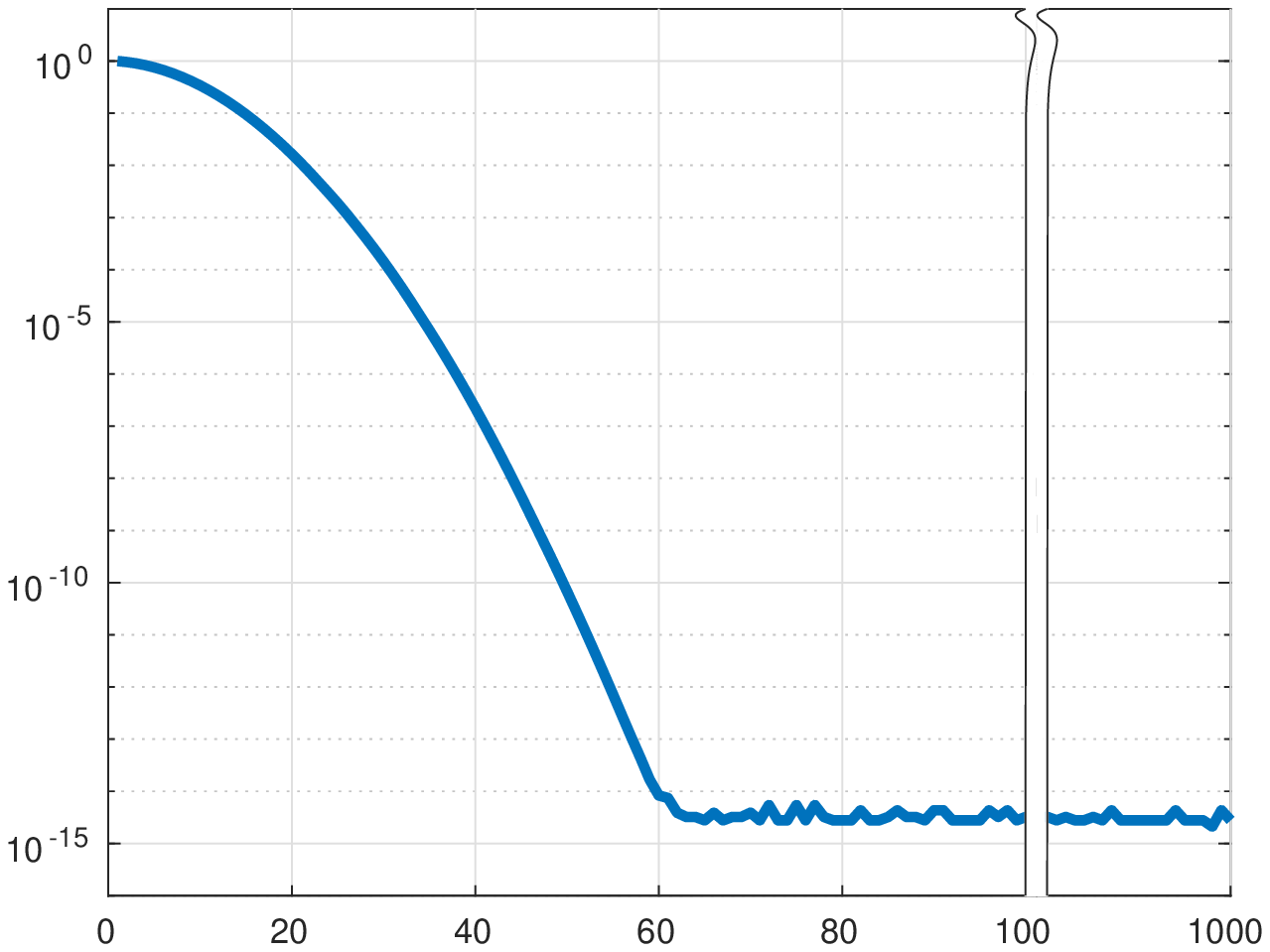}
\put(-77,106){\scriptsize Error}
 \put(-72,-5){\scriptsize $N$}
 \put(-370,106){\scriptsize MATLAB {\tt spy} plot}
 \put(-238,106){\scriptsize Solution to~(\ref{eqn:ex5})}
 \put(-212,-5){\scriptsize $t$}
 \put(-291,48){\scriptsize \rotatebox{90}{$y(t)$}}
 \caption{Left: MATLAB \lstinline{spy} plot of the $N = 60$ discretisation of the operator in~(\ref{eqn:ex5spy}). 
 The operator is `almost-banded' (with one dense row at the top due to the boundary condition) and the 
 bandwidth is determined by the decay in the Legendre coefficients of the kernel.
 Middle: Solution (solid line = linear scale, dashed line = $\log_{10}$ in $y$) of~(\ref{eqn:ex5}) with $a = 100$,
 evaluated using Clenshaw's algorithm. 
 Right: Error (approx infinity norm at 1000 equally-spaced points) for various values of $N$. As the solution $y(t)$
 is entire we observe super-geometric convergence. The well-conditioned discretisation of the differentiation operator 
 means there is no growth in the error even for very large values of $N$.} \label{fig:example5}
 \end{figure}

%

\end{example}

\tldr{Combining the Legendre-US method with the discrete Legendre convolution operator from Section~\ref{subsec:convop} 
allows the solution of VIDECTS. When the kernel is polynomial (or sufficiently well-approximated by a polynomial)
then the resulting discretisations are banded and the method is spectrally accurate.}

\subsection{Solving Fredholm IDEs}
We saw in the previous section how an ultraspherical spectral method for VIDEs might be derived.
However, VIDEs are typically initial value problems, for which local or time-stepping discretisations are more commonly used. 
Fredholm operators, on the other hand, are themselves global in nature, and hence a global spectral discretisation of FIDEs
is more natural. To this end, consider now the basic Fredholm operator with a convolution kernel, which we may write as
\be\label{eqn:fredop}
({\cal{F}}[k]y)(t) := \int_{0}^Tk(t-s)y(s)\, ds = ({\cal{V}}[k]y)(t) + (\widetilde {\cal{V}}[k]y)(t)
\ee
where 
\be
(\widetilde {\cal{V}}[k]y)(t) := \int_t^T k(t-s)y(s)\, ds.
\ee
Making the substitution $\tau = T - t$, $\sigma = T - s$ gives
\be
(\widetilde {\cal{V}}[k]y)(t) = \int_0^{\tau} k(\sigma - \tau)y(T - \sigma)\, d\sigma = \int_0^{\tau} \widetilde k(\tau - \sigma)\widetilde y(\sigma)\, d\sigma = ({\cal{V}}[\widetilde k]\widetilde y)(\tau), \quad 0 \le t,\tau \le T,
\ee
where $\widetilde y(\tau) = y(T-t)$ and $\widetilde k(t) = k(-t)$. Since $\widetilde P_n(-x) = (-1)^n \widetilde P_n(x)$ for Legendre polynomials on $[-1, 1]$, 
we have that $P_n(T-t) = (-1)^n P_n(t)$ for mapped-Legendre polynomials on $[0, T]$, and hence that 
$\widetilde{\underline{y}} = \widetilde I{\underline{y}}$ and ${\underline{y}} = \widetilde I\widetilde{\underline{y}}$
where 
\be
\widetilde I := \begin{bsmallmatrix}1 \\ & -1 \\ && \,\,\,1\\ &&& -1\\ &&&&\ddots\\\end{bsmallmatrix}.
\ee
Therefore, the discrete version of $\widetilde{\cal{V}}[k]$ may be written as
\be
\widetilde V[\underline{k}]\underline{y} = \widetilde I V[\underline{\widetilde k}]\widetilde I \underline{y},
\ee
where $\widetilde{\underline{k}}$ are the Legendre coefficients of $k(-t)$ in $[0, T]$, 
and we may express the Fredholm operator~(\ref{eqn:fredop}) as
\be
({\cal{F}}[k]y)(t) = \P(t)F[\underline{k}, \widetilde{\underline{k}}]\underline{y} := \P(t)\left(V[\underline{k}] + \widetilde I V[\widetilde{\underline{k}}]\widetilde I\right)\underline{y}.
\ee
Like the discrete Volterra operator in the previous section, $F[\underline{k}, \widetilde{\underline{k}}]$ will be banded if $k(t)$ is a polynomial 
(or sufficiently well-approximated by a polynomial on $[-T,0]$ and $[0,T]$) and can be combined with the Legendre-US method to solve Fredholm integro-differential equations. 
The derivation above is readily extended to Fredholm operators with non-constant coefficients, which may be written as
$$
({\cal{F}}[k]y)(t) = g(t)\int_{0}^Tk(t-s)h(s)y(s)\, ds = \P(t)M_{1/2}[\underline{g}]\left(V[\underline{k}] + \widetilde I V[\widetilde{\underline{k}}]\widetilde I\right)M_{1/2}[\underline{h}]\underline{y}.
$$

\begin{example}\label{example:7}
We modify the example from the previous section so that the solution~(\ref{eqn:ex5sol}) to~(\ref{eqn:ex5}) is also the solution to the second-order Fredholm integro-differential equation
 \be\label{eqn:ex7}
 y''(t) + ay'(t) - y(t) = f(t) - \int_0^Te^{-(t-s)}y(s)\,ds, \quad 0\le t \le 1, \quad y(0) = 1, \ \  y(T) = \gamma_T,
 \ee
 where $f(t) = \exp(-t)\big(\exp\big(\tfrac{1-a}{2}T\big)\sinh(bT) - \exp\big(\frac{1-a}{2}t\big)\sinh(bt)\big)/b$, $a$ is a given constant, and $\gamma_T$ is obtained by evaluating~(\ref{eqn:ex5sol}) at $t = T$. 
 Using the approach described above, our discretisation is given by
 \be\label{eqn:ex7spy}
 \begin{bmatrix}
 \begin{bmatrix}
 1,&-1,&1,&-1,&1&-1,&1&-1,&\ldots\\
 1,&1,&1,&1,&1,&1,&1,&1,&\ldots
 \end{bmatrix}\\
 D^2_{1/2} + S_{3/2}(aD_{1/2} + S_{1/2}(-I+F[\underline{e}^{-t}, \underline{e}^{t}]))\end{bmatrix}\underline{y} = \begin{bmatrix}\begin{bmatrix}1\\\gamma_T\end{bmatrix}\\S_{3/2}S_{1/2}\underline{f}\end{bmatrix}.
 \ee
 Figure~\ref{fig:example7} shows the resulting discretisation, solution, and convergence when $a = 100$. 
 There is very little difference to the corresponding results for~(\ref{eqn:ex5}),
 other than that the two boundary conditions now give rise to two dense rows at the top of the matrix.
 As before, the discretisation is almost-banded and the convergence super-geometric.%
\begin{figure}[t]
 \includegraphics[height=105pt, trim={0 15pt 0 0},clip]{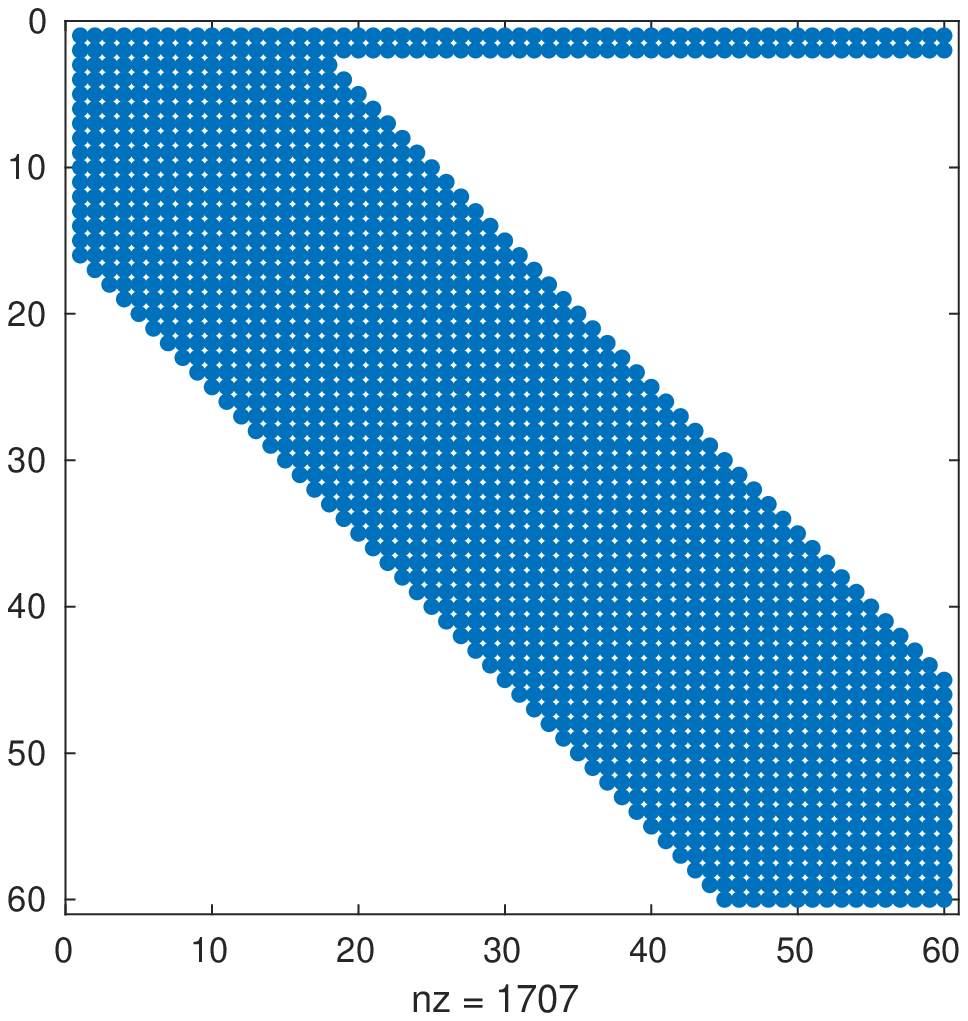}\hfill\hspace*{5pt}
 \includegraphics[height=105pt]{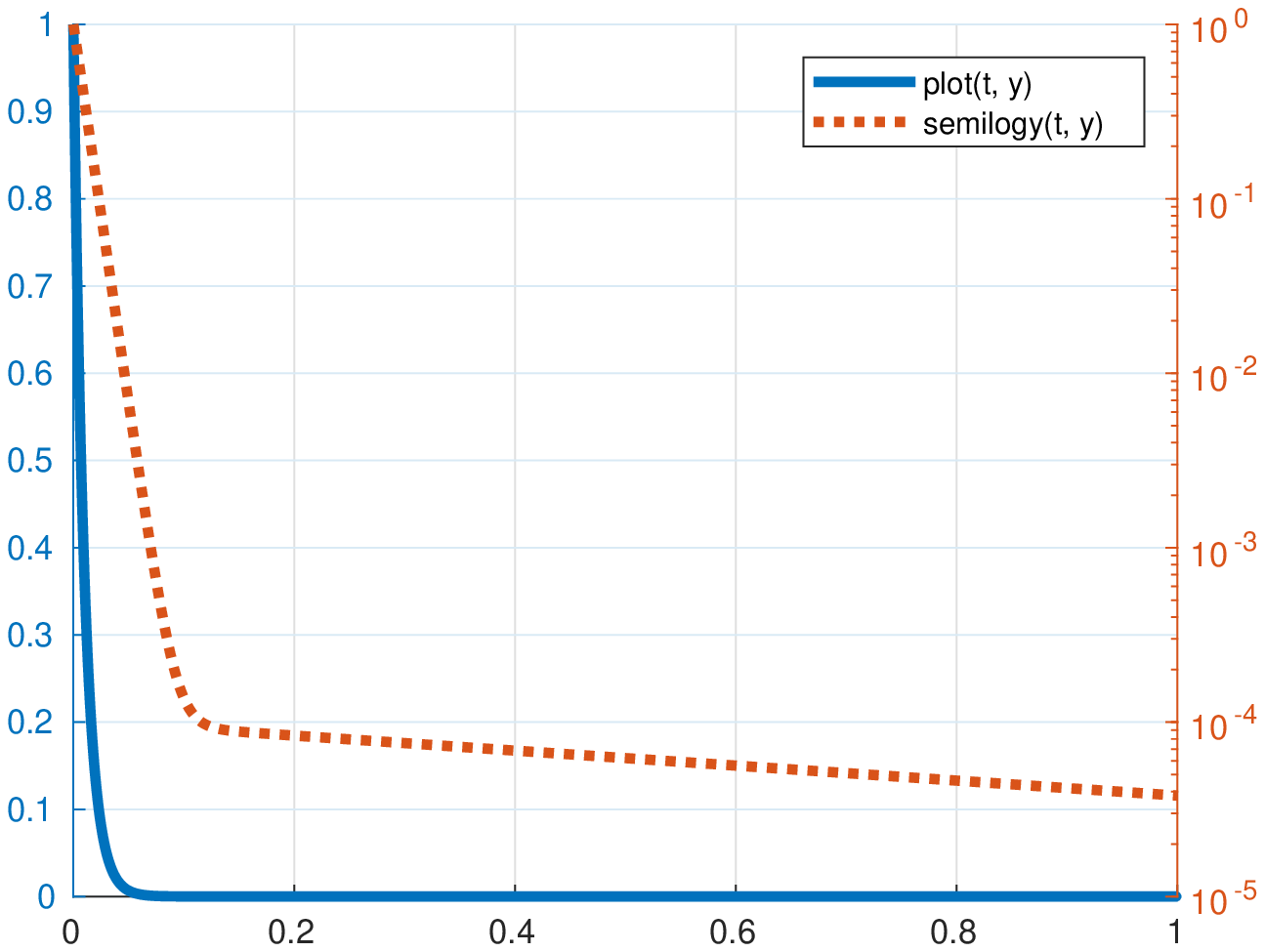}\hfill
 \includegraphics[height=105pt]{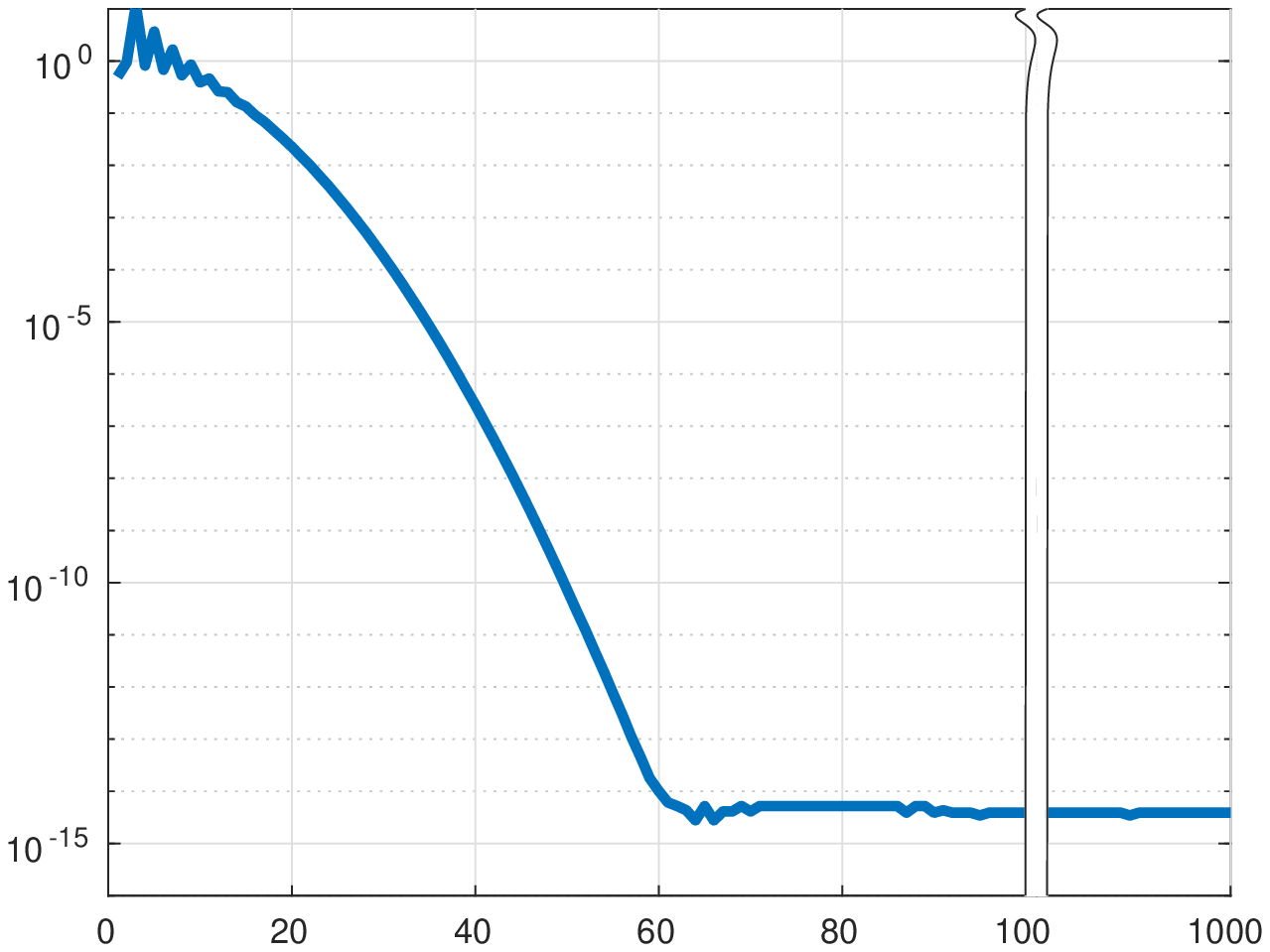}
 \put(-77,106){\scriptsize Error}
 \put(-72,-5){\scriptsize $N$}
 \put(-370,106){\scriptsize MATLAB {\tt spy} plot}
 \put(-238,106){\scriptsize Solution to~(\ref{eqn:ex7})}
 \put(-212,-5){\scriptsize $t$}
 \put(-291,48){\scriptsize \rotatebox{90}{$y(t)$}}
 \caption{As in Figure~\ref{fig:example5}, but here for the FIDECT~(\ref{eqn:ex7}). 
 The only significant difference in the discretisations is that the operator now has two dense rows at the top, because of the two boundary conditions.} \label{fig:example7}
 \end{figure}

\end{example}

\textbf{Remark:} Note that the splitting in~(\ref{eqn:fredop}) means we require only that the kernel is {\em piecewise} smooth on $[-T, 0]$ and $[0,T]$, not smooth on $[0,T]$. In particular, 
this allows kernels of the form $k(|t-s|)$ since here
\be\label{eqn:fredopabs}
\int_{0}^Tk(|t-s|)y(s)\, ds = \underbrace{\int_{0}^tk(t-s)y(s)\, ds}_{({\cal{V}}[k]y)(t) } + \underbrace{\int_{t}^Tk(s-t)y(s)\, ds}_{({\cal{V}}[\widehat k]y)(t) },
\ee
from which is follows that the discrete version of the operator is given by $F[\underline{k},\underline{k}] = V[\underline{k}] + \widetilde I V[\underline{k}] \widetilde I $.

\tldr{Fredholm operators can be written as the sum of a Volterra operator and a `flipped' Volterra, allowing the Legendre convolution operator to be used for Fredholm IDEs.}

\section{Further examples}\label{sec:examples}%
The two examples considered in the previous section were intentionally elementary. 
In fact, in both cases the rank-1 kernel, $k(t,s) = e^{-(t-s)}$, could be dealt with 
by a variant of the degenerate kernel approach described in Section~\ref{subsec:us}. In this section 
we consider some more challenging examples.

\begin{example}\label{ex:4}
Gaussian kernels arise often in Volterra and Fredholm integral equations with applications in filtering and scattering. 
These functions are not degenerate and such equations are therefore not so easily solved with standard techniques as those with, say, 
exponential kernels in the examples above. However, these kernels are smooth and of convolution type, so the method of this 
paper is readily applicable. To this end, consider the second-order Fredholm IDE given by 
\begin{equation}\label{eqn:ex4}
\xi^2y''(t) + ty'(t) + y(t) + \int_0^1 e^{-\frac{(t-s)^2}{2\sigma^2}}y(s)\, ds = f(t), \quad 0 \le t \le 1.
\end{equation}
For convenience we choose the function $f(t)$ on the right hand-side of~(\ref{eqn:ex4}) as 
\begin{equation}
 f(t) =  \frac{\xi\sigma}{r}\sqrt{\frac{\pi}{2}} e^{-\frac{t^2}{2r^2}} \left(\erf\left(\frac{\xi t}{\sigma r\sqrt{2}}\right) + \erf\left(\frac{r^2-\xi^2t}{\xi\sigma r\sqrt{2}}\right)\right),
\end{equation}
where $r = \sqrt{\sigma^2+\xi^2}$ and enforce the additional constraints 
\begin{equation}\label{eqn:ex4bcs}
y(0) = 1 \quad \text{and} \quad \int_0^1y(t) dt = \sqrt\frac{\pi}{2}\xi\erf\left(\frac{1}{\sqrt{2}\xi}\right) =: \gamma_\xi,
\end{equation}
so that the solution to~(\ref{eqn:ex4}) is given by 
\begin{equation}
 y(t) = e^{-\frac{t^2}{\xi^2}}.
\end{equation}

To discretise this FIDE, we first note that the mean-value constraint in~(\ref{eqn:ex4bcs}) can be enforced by observing that if $P_n(t)$ is a mapped-Legendre polynomial on $[0, T]$ then
\begin{equation}
 \int_{0}^TP_n(t)\,dt = \left\{\begin{array}{ll}T,& n=0,\\0, & \text{otherwise,}\end{array}\right.
\end{equation}
so that here we require $[1, 0, 0,\ldots ]\underline{y} = \gamma_\xi$.
The additional complication in this example is the non-constant coefficient in front of the $y'(t)$ term.
Since $D_{1/2}:\P\rightarrow\mathbf{C}^{3/2}$, the multiplication must take place in $\mathbf{C}^{3/2}$.
To this end, we construct the mapped-Legendre coefficients of $t$ on $[0, 1]$ (i.e., $\underline{t} = [\tfrac12, \tfrac12, 0, \ldots]$)
and convert these to coefficients in $\mathbf{C}^{3/2}$ by multiplying with the conversion operator $S_{1/2}$ (i.e., $S_{1/2}\underline{t} = [\tfrac12, \tfrac16, 0, \ldots]$.
The resulting multiplication operator $M_{3/2}[S_{1/2}\underline{t}]$ can then be constructed using the technique described in~\cite[Section 2.3]{hale2017}.
In this simple case it is given by
\be
M_{3/2}[S_{1/2}\underline{t}] = 
\tfrac{1}{2}
\begin{bsmallmatrix}
 1 \ & \frac{3}{5} \ \\
 \frac{1}{3}\ & 1\ & \frac{4}{7}\\
 & \frac{2}{5}\ & 1\ & \frac{5}{9}\\
   &  & \frac{3}{7}\ & 1\ & \ddots\\
   & & & \ddots & \ddots
\end{bsmallmatrix}
= \tfrac{1}{2}(I + J_{3/2}),
\ee
where $J_{3/2}$ is given by~\cite[(2.15)]{hale2017} with $\lambda = 3/2$. It remains then 
to compute the mapped-Legendre coefficients $\underline{f}$ of $f(t)$, $\underline{k}$ 
of the kernel $k(t) = e^{-{t^2}/{2\sigma^2}}$, and $\underline{\widehat k}$ 
of the flipped kernel $\widehat{k}(t) = e^{{t^2}/{2\sigma^2}}$
on the interval $[0, 1]$  so that
the discretised version of~(\ref{eqn:ex4}) and~(\ref{eqn:ex4bcs}) is given by
 \be\label{eqn:ex4spy}
 \begin{bmatrix} 
 \begin{bmatrix}
 1,&-1,&1,&-1,&1,&-1,&1,&-1,&-1,&\ldots\\
 1,&0,&0,&0,&0,&0,&0,&0,&0,&\ldots
 \end{bmatrix}\\
 D^2_{1/2} + S_{3/2}(M_{3/2}[S_{1/2}\underline{t}]D_{1/2}+S_{1/2}(I + F[\underline{k}, \underline{\widehat k}]))\end{bmatrix}\underline{y} = \begin{bmatrix}\begin{bmatrix}1\\\gamma_\xi\end{bmatrix}\\S_{3/2}S_{1/2}\underline{f}\vphantom{D^2_{1/2}}\end{bmatrix}.
 \ee
 
 Figure~\ref{fig:example4} shows the results when $\xi = {1}/{10}$ and $\sigma = 1$. As usual in the first panel we see that the 
 discretisation is almost-banded, here with only one dense row since the mean value constraint in~(\ref{eqn:ex4bcs}) is sparse.
 The bandwidth of the operator is determined by $\sigma$ in the Gaussian kernel; making $\sigma$ larger will
 result in a wider bandwidth, whilst decreasing it will make the bandwidth narrower. Conversely, the rate of convergence 
 (right-most figure) depends only on the value of $xi$; smaller $\xi$ will require a higher $N$ to resolve.%
\begin{figure}[t]
 \includegraphics[height=105pt, trim={0 15pt 0 0},clip]{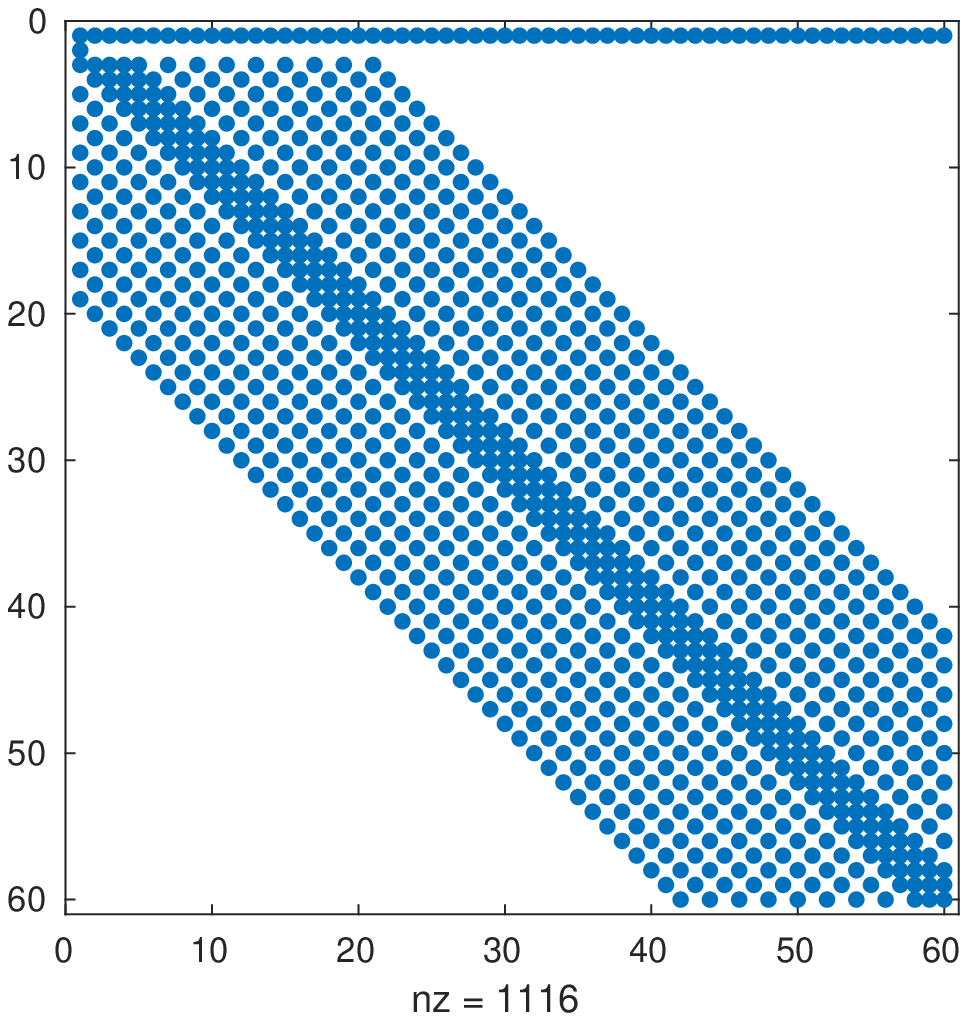}\hfill\hspace*{5pt}
 \includegraphics[height=105pt]{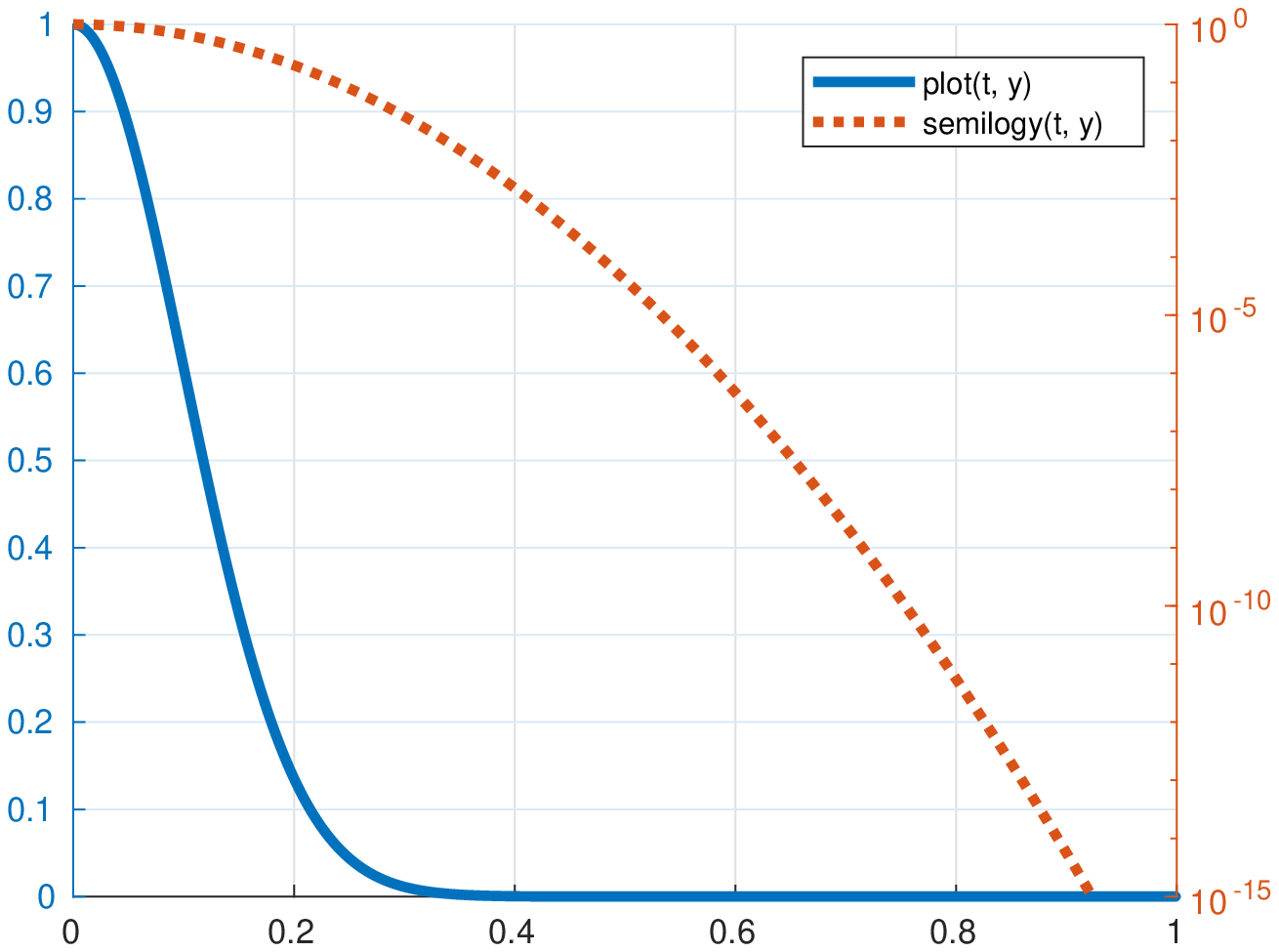}\hfill
 \includegraphics[height=105pt]{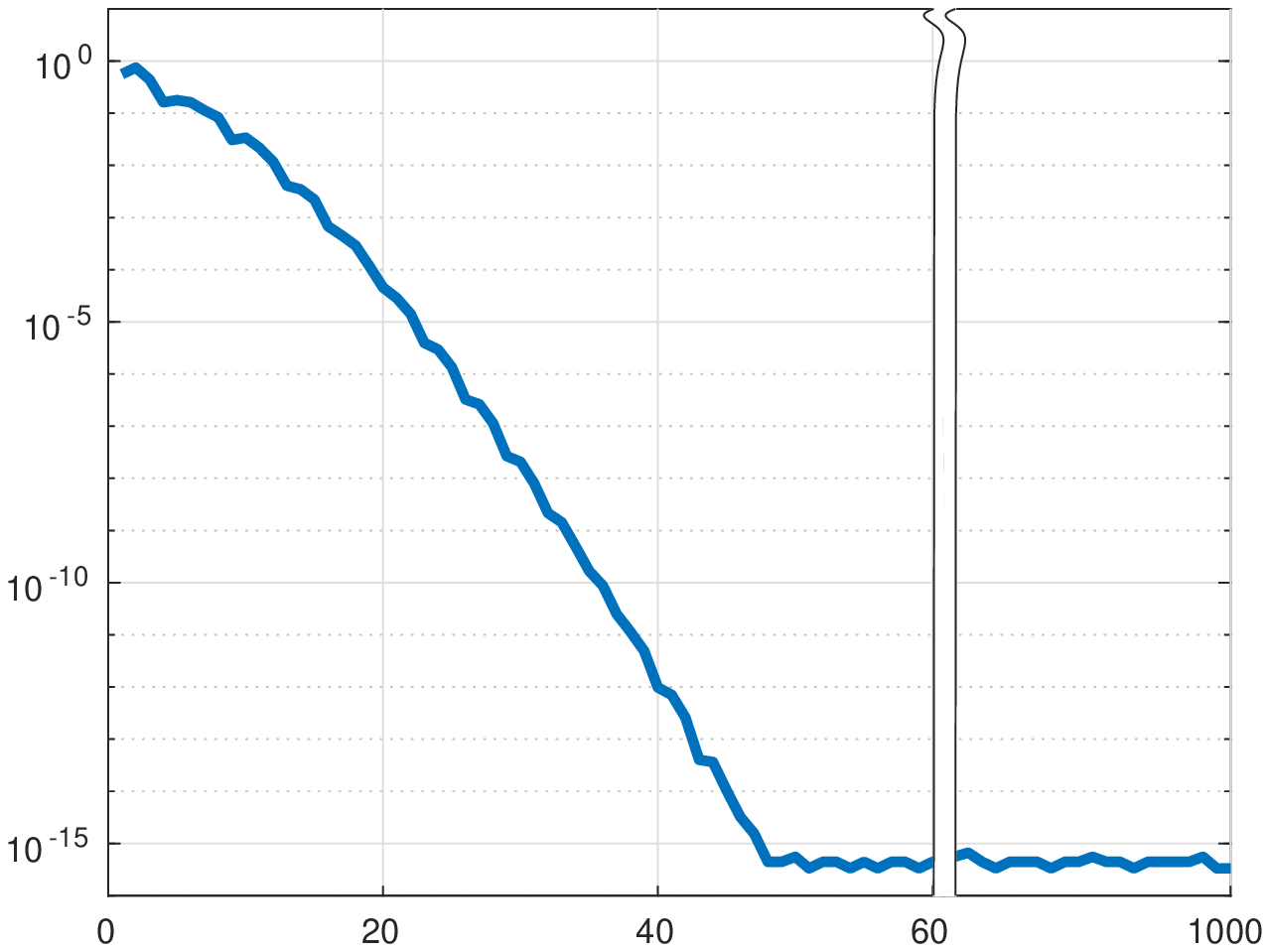}
 \put(-77,106){\scriptsize Error}
 \put(-72,-5){\scriptsize $N$}
 \put(-370,106){\scriptsize MATLAB {\tt spy} plot}
 \put(-238,106){\scriptsize Solution to~(\ref{eqn:ex4})}
 \put(-212,-5){\scriptsize $t$}
 \put(-291,48){\scriptsize \rotatebox{90}{$y(t)$}}
 \caption{Discretisation, solution, and convergence for the second-order FIDECT~(\ref{eqn:ex4}).
 As in the previous examples, the almost-banded structure and exponential convergence of the solution 
 are still obtained, even in the presence of non-constant coefficients and more general boundary conditions.} \label{fig:example4}
 \end{figure}

\end{example} 

\begin{example}
As our final example, for any $\omega\in\mathbb{R}$ and non-negative integer $\mu>0$, consider the second-order VIDECT
\begin{equation}\label{eqn:ex13}
  y''(t) + \omega^2y(t) = f(t) -\omega\int_0^t J_\mu(\omega(t - s))y(s)\,ds  , \quad 0\le t \le T, \quad y(0) = y'(0) = 0,
\end{equation}
where $f(t)$ on the right hand-side is taken as 
\begin{equation}
  f(t) = J_{\mu+\eta}(\omega t) + \frac{1}{2t^2}\big((\eta-1)(\eta-2)J_{\eta-1}(\omega t)+(\eta+1)(\eta+2)J_{\eta+1}(\omega t)\big)
\end{equation}
so that for any integer $\eta \ge 3$ the exact solution is given by\footnote{In fact, if the initial conditions are changed to $y(0) = {\delta_{1,\eta}}/{2}, \, y'(0) = {\omega\delta_{2,\eta}}/{4}$
then the result holds also for $\eta = 1$ and $\eta = 2$. The solution to this VIDECT was reverse engineered using the 
differential equation satisfied by Bessel function~\cite[10.2.1]{DLMF} and the known convolution of two Bessel functions~\cite[10.22.34]{DLMF}.}
\begin{equation}
  y(t) = \eta J_\eta(\omega t)/(\omega t).
\end{equation}
Integral equations with Bessel function kernels arise often in the fields of scattering and potential theory~\citep{xiang2013},
and since $J_\mu(\omega(t - s))$ is not degenerate, the Volterra operator in~(\ref{eqn:ex13}) cannot trivially be written 
as a linear combination of multiplication and integration operators.

We proceed as before, discretising the differential operator ${\cal{D}}^2+\omega^2{\cal{I}}$ using the US method for ODEs, 
and the convolution-type Volterra operator ${\cal{V}}[J_\mu(\omega t)]$ using the approach outlined in Section~\ref{subsec:convop}.
The Dirichlet boundary condition $y(0)= 0 $ can be enforced as in~(\ref{eqn:ex5spy}), and for the Neumann condition $y'(0)= 0 $ 
we have that
\begin{equation}
  y'(0) = \P'(0)\underline{y} = T\begin{bmatrix}0,&2,&-6,&12,&-20,&\ldots\end{bmatrix}\underline{y} = 0,
\end{equation}
The full discretisation of~(\ref{eqn:ex13}) on the interval $[0, 1]$ is then given by 
 \be\label{eqn:ex13spy}
 \begin{bmatrix} 
 \begin{bmatrix}
 1,&-1,&1,&-1,&1,&\ldots\\
 0,&2,&-6,&12,&-20,&\ldots
 \end{bmatrix}\\
 D^2_{1/2} + S_{3/2}S_{1/2}(\omega^2I + \omega V[\underline{J}_\mu\!(\omega t)])\end{bmatrix}\underline{y} = \begin{bmatrix}\begin{bmatrix}0\\0\end{bmatrix}\\\underline{f}\vphantom{D^2_{1/2}}\end{bmatrix}.
 \ee
\begin{figure}[t]
 \includegraphics[height=105pt, trim={0 15pt 0 0},clip]{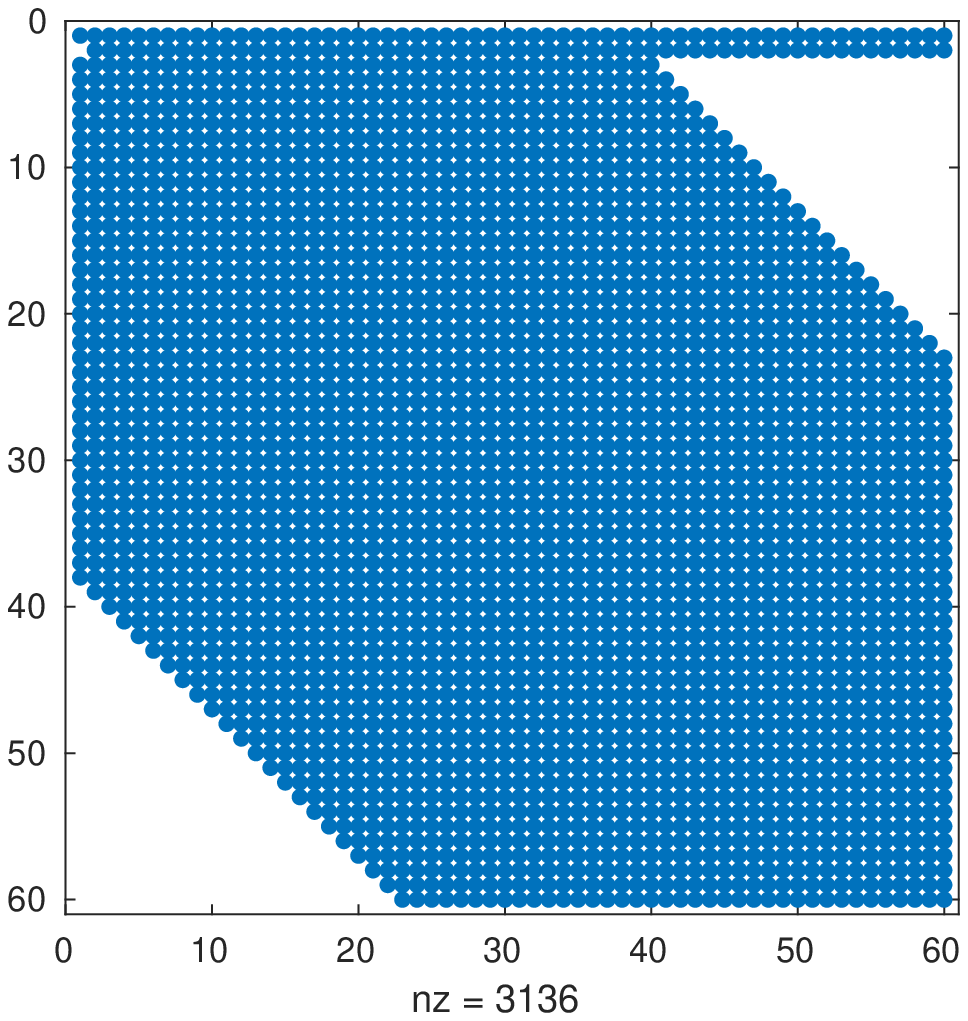}\hfill\hspace*{5pt}
 \includegraphics[height=105pt]{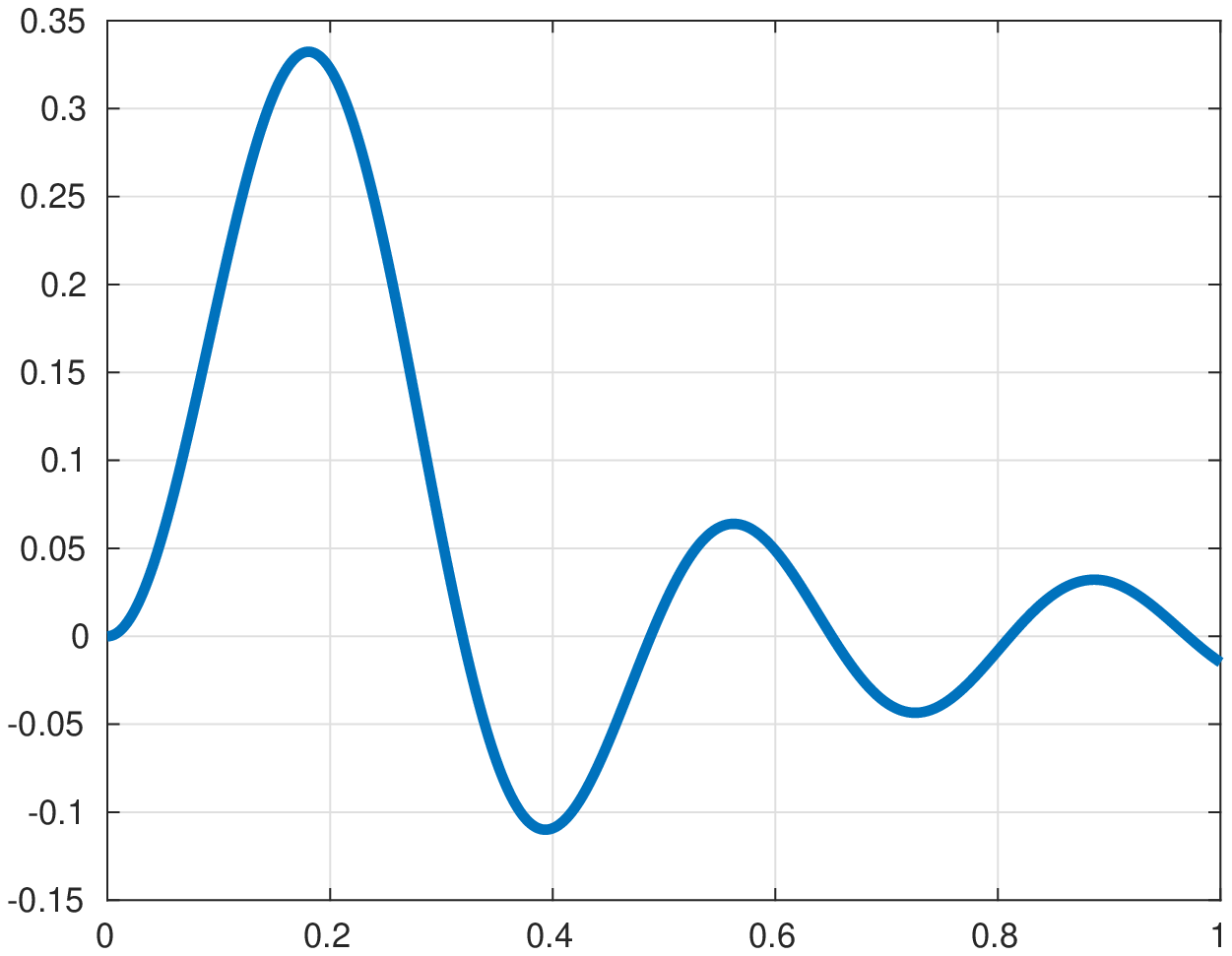}\hfill
 \includegraphics[height=105pt]{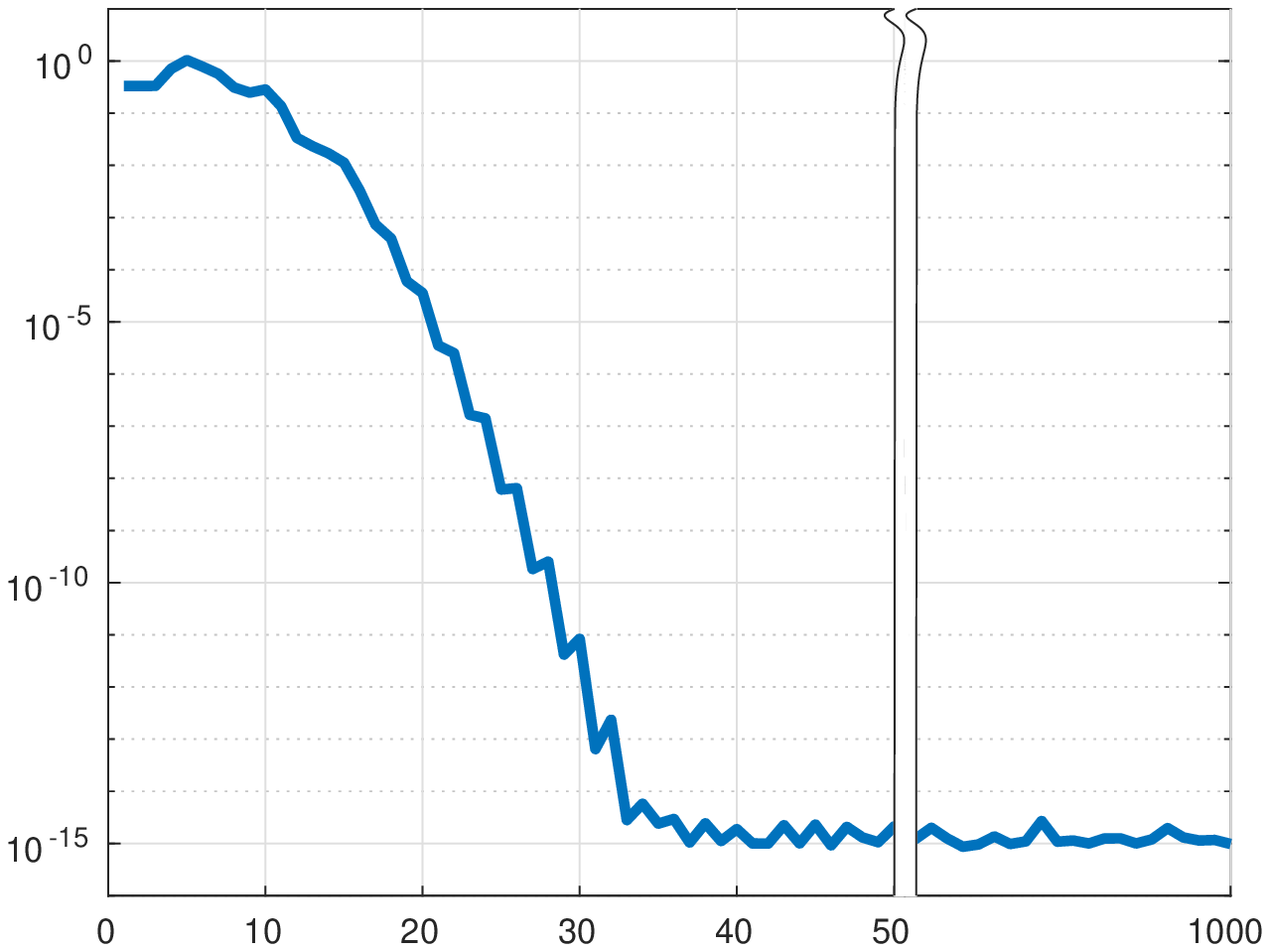}
\put(-77,106){\scriptsize Error}
 \put(-70,-5){\scriptsize $N$}
 \put(-370,106){\scriptsize MATLAB {\tt spy} plot}
 \put(-228,106){\scriptsize Solution to~(\ref{eqn:ex13})}
 \put(-206,-5){\scriptsize $t$}
 \put(-291,48){\scriptsize \rotatebox{90}{$y(t)$}}
 \caption{Discretisation, solution, and convergence for the second-order VIDECT~(\ref{eqn:ex13}) with $\nu = 3$, $\mu = 2$, and $\omega = 20$.
 Geometric convergence is obtained (right-most panel) to the Bessel function-like solution (centre panel).
 However, although almost-banded as $N\rightarrow\infty$, this structure is not utilised in practice for this
 example since the frequency of the solution is the same as that kernel (left-most panel). One can construct am
 example for which this isn't the case by adding a singular perturbation to the left hand-side (see Figure~\ref{fig:example13b}).} \label{fig:example13}
\end{figure}

Figure~\ref{fig:example13} shows the discretisation, solution, and error when solving~(\ref{eqn:ex13}) with $\eta = 3, \mu = 2$, and $\omega = 20$.
As in the previous examples, the first panel shows the almost-banded structure of the discretisation,
here with two dense rows arising from the two initial conditions. The bandwidth is determined by
the frequency $\omega$ in the kernel which, by construction, is the same as the frequency of the solution, and so the banded 
nature of the operator (in the limit as $N\rightarrow\infty$) is not utilized in practice for this example. 
However, one can readily imagine related problems which require larger values of $N$ to solve (for example, 
introducing small perturbation parameter in front of the $y''$ term or adding an oscillatory or less smooth term to $f(t)$)
that will not affect the structure of the operator. Figure~\ref{fig:example13b} shows such a situation, 
where the left hand-hand side of~(\ref{eqn:ex13}) is modified to $0.001 y''(t) + \omega^2y(t)$.
This singular perturbation causes the solution to become far more oscillatory, and hence far more degrees
of freedom are required to resolve it. In this case, the banded nature of the discretisation is evident, and utilised.
In both cases, exponential convergence and high accuracy is observed to around the level of machine precision.%
\begin{figure}[t]
 \includegraphics[height=105pt, trim={0 15pt 0 0},clip]{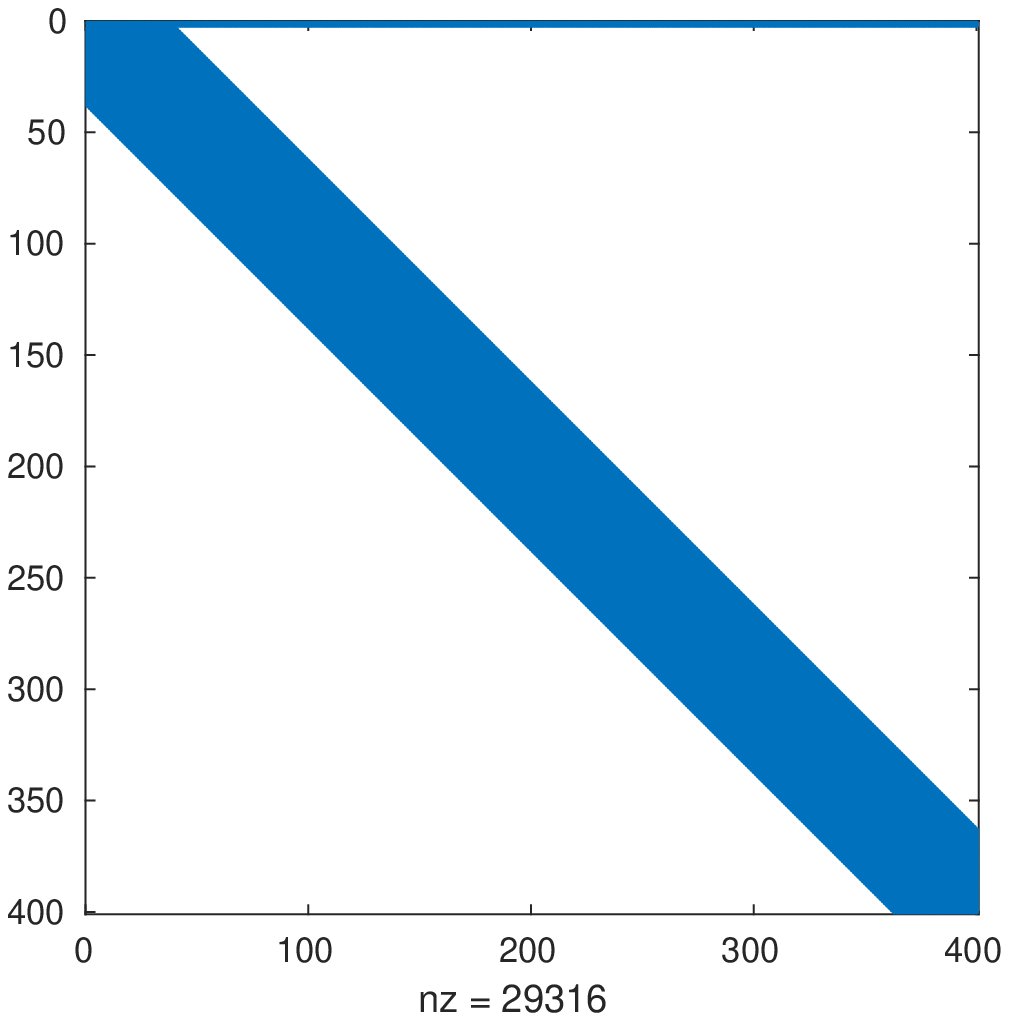}\hfill\hspace*{5pt}
 \includegraphics[height=105pt]{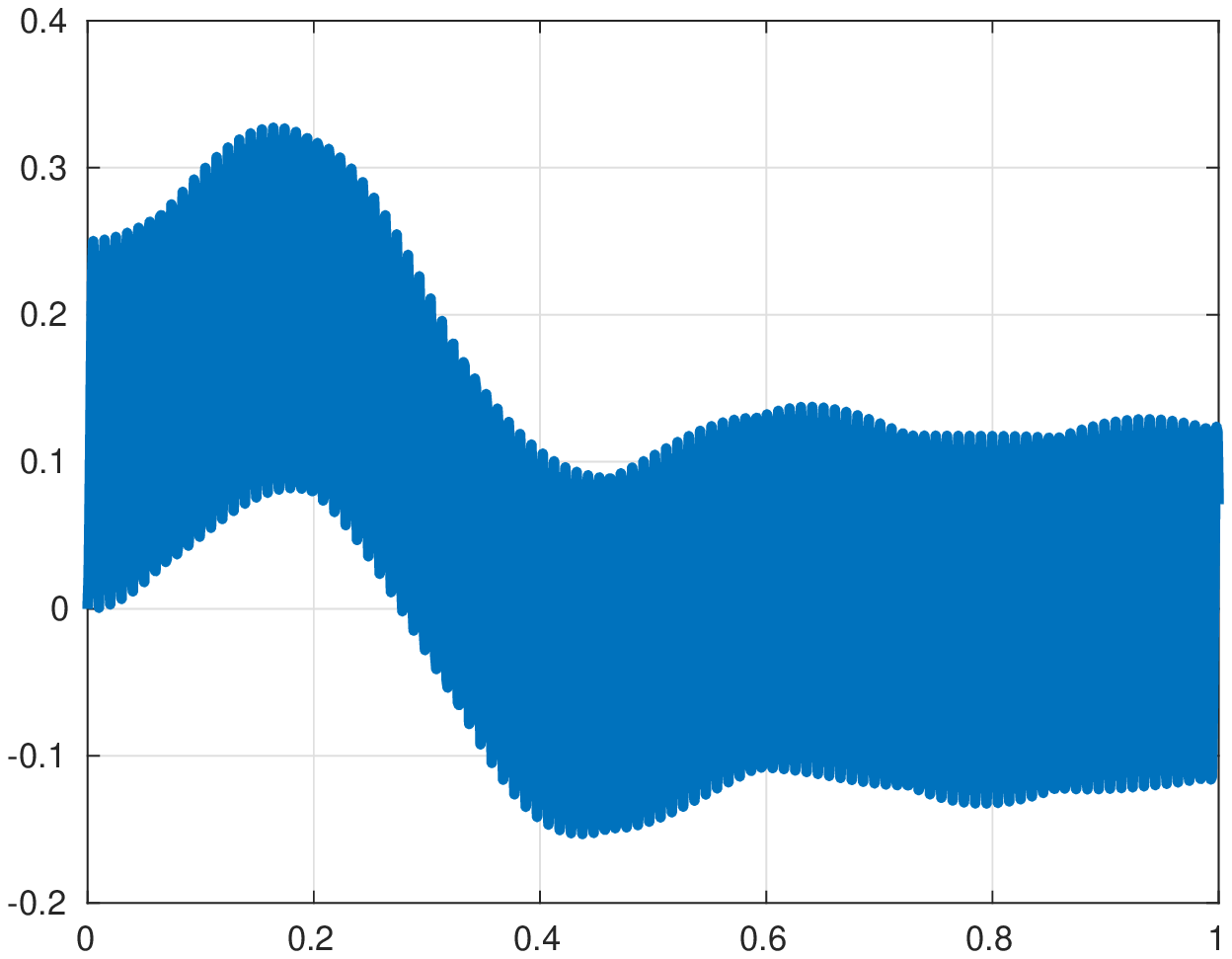}\hfill
 \includegraphics[height=105pt]{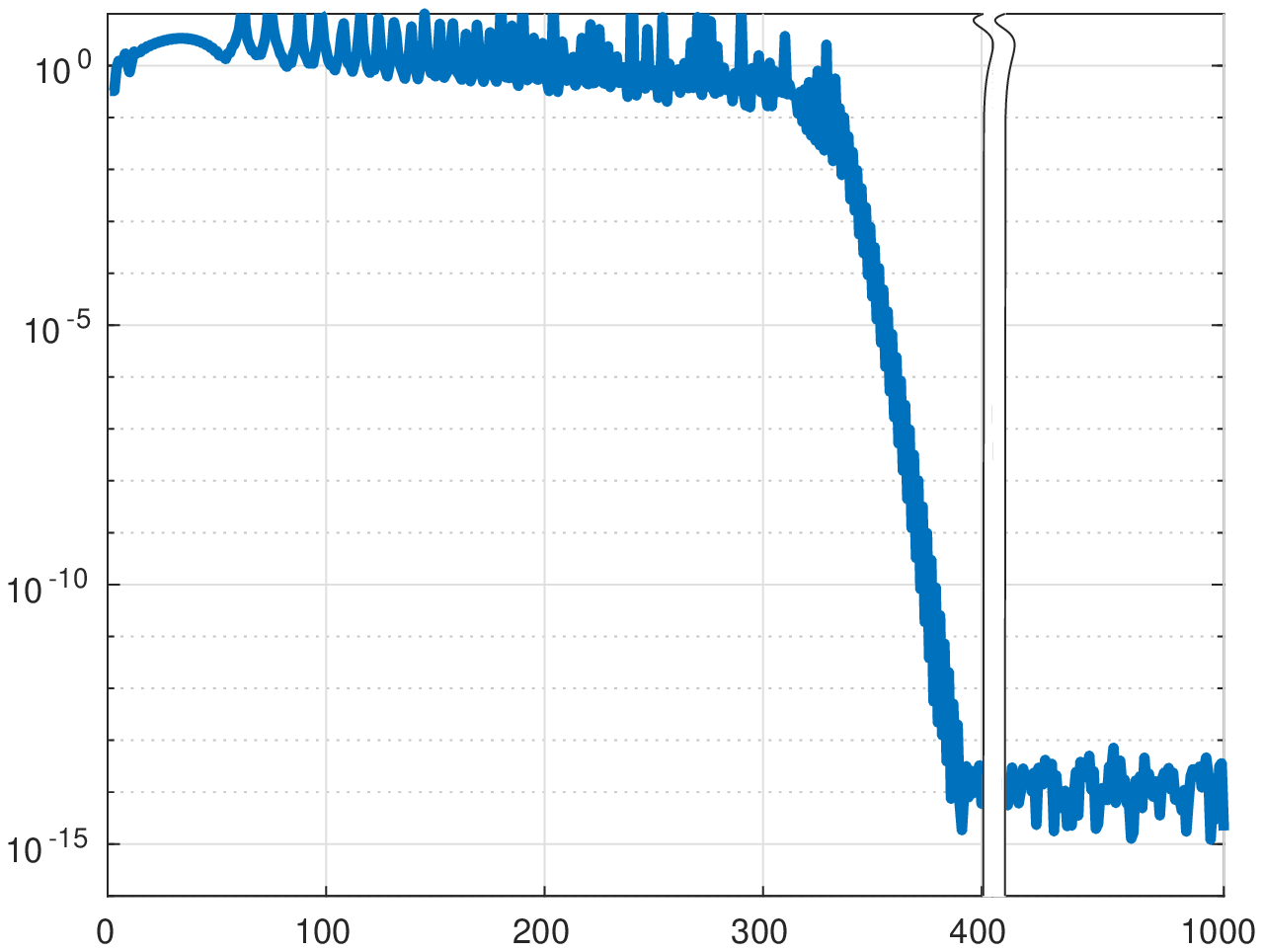}
\put(-77,106){\scriptsize Error}
 \put(-70,-5){\scriptsize $N$}
 \put(-370,106){\scriptsize MATLAB {\tt spy} plot}
 \put(-242,106){\scriptsize Solution to modified~(\ref{eqn:ex13})}
 \put(-206,-5){\scriptsize $t$}
 \put(-291,48){\scriptsize \rotatebox{90}{$y(t)$}}
 \caption{As in Figure~\ref{fig:example13} but with the left hand-side of~(\ref{eqn:ex13}) modified to $0.001 y''(t) + \omega^2y(t)$.
 The singular perturbation makes the solution far more oscillatory, and the banded nature of the discretisation is utilised.
 Here the underlying solution is not known, so the error is computed against a discretisation with $N=2000$.
 We see in the right-most panel that, once the oscillations are resolved, convergence is again geometric.} \label{fig:example13b}
 \end{figure}
 \end{example}
%
\section{Conclusion}\label{sec:conc}%
We have shown that the linear operators representing the 
action of Fredholm and Volterra integral operators on Legendre
series expansions can be combined with the Legendre-ultraspherical spectral
method for ordinary differential equations to efficiently solve 
integro-differential equations with convolution-type kernels. 
It was demonstrated that if the kernel, coefficient functions, and right-hand side 
are sufficiently smooth (analytic) then the rate of convergence is very fast (geometric).
If the kernel and coefficient functions are polynomials (or sufficiently
well-approximated by a a polynomial) then the discretisations result in 
banded or almost-banded matrices which can be solved with linear complexity.

Volterra equations typically correspond to initial value problems, 
and hence most numerical techniques use some form of time-marching.
However, the geometric convergence and linear complexity of the 
algorithm presented in this paper make it competitive with 
such approaches in many cases. Fredholm integral operators are like 
boundary-value problems and hence a global discretisation is more natural.

We presented here the simplest version of this algorithm, namely linear
problems in one dimension. The approach readily extends to 
time-dependent problems, systems of equations, and higher-dimensions, analogously to the
ultraspherical method for differential equations~\citep{olver2014}. Nonlinear
problems may be solved by a Newton--Raphson-like approach, but
in such situations the banded structure (although not the well-conditioned)
property of the discretisation is lost. A linear complexity method
for nonlinear problems, even for ODES, remains elusive.
\section*{Acknowledgements}
I am grateful to  Sheehan Olver (Imperial College), Alex Townsend (Cornell), and Andr\'e Weideman (Stellenbosch University) for useful discussions.

\bibliographystyle{IMANUM-BIB}
\bibliography{Hale2017}

\begin{thebibliography}{}

\bibitem[Alpert \& Rokhlin(1991)Alpert \& Rokhlin]{alpert1991}
{\sc Alpert, B.~K. \& Rokhlin, V.} (1991)
\newblock A fast algorithm for the evaluation of {L}egendre expansions.
\newblock {\em SIAM J. Sci. Statist. Comput.}, {\bf 12}, 158--179.

\bibitem[Cox(1962)Cox]{cox1962}
{\sc Cox, D.~R.} (1962)
\newblock {\em Renewal theory\/}.
\newblock Methuen \& Co. Ltd., London; John Wiley \& Sons, Inc., New York, pp.
  ix+142.

\bibitem[{\relax DLMF}(2015){\relax DLMF}]{DLMF}
{\relax DLMF}
\newblock  (2015).
\newblock \url{http://dlmf.nist.gov/}, Release 1.0.10 of 2015-08-07.

\bibitem[Driscoll(2010)Driscoll]{driscoll2010}
{\sc Driscoll, T.~A.} (2010)
\newblock {Automatic spectral collocation for integral, integro-differential,
  and integrally reformulated differential equations}.
\newblock {\em J. Comput. Phys.}, {\bf 229}, 5980--5998.

\bibitem[Driscoll {\em et~al.}(2014)Driscoll, Hale, \& Trefethen]{Chebfun}
{\sc Driscoll, T.~A., Hale, N. \& Trefethen, L.~N.} (2014)
\newblock {\em Chebfun Guide\/}.
\newblock Pafnuty Publications.

\bibitem[El-gendi(1970)El-gendi]{elgendi1969}
{\sc El-gendi, S.~E.} (1969/1970)
\newblock Chebyshev solution of differential, integral and integro-differential
  equations.
\newblock {\em Comput. J.}, {\bf 12}, 282--287.

\bibitem[Hale(2017)Hale]{hale2017_code}
{\sc Hale, N.}
\newblock  (2017).
\newblock \url{https://github.com/nickhale/usconv}.
\newblock Last accessed 01 Dec 2017.

\bibitem[Hale \& Olver(2017)Hale \& Olver]{hale2017}
{\sc Hale, N. \& Olver, S.} (submitted Nov 2017)
\newblock A fast and spectrally convergent algorithm for fractional integral
  and differential equations with half-integer order terms.
\newblock {\em SIAM J. Sci. Comput.}

\bibitem[Hale \& Townsend(2014a)Hale \& Townsend]{hale2014b}
{\sc Hale, N. \& Townsend, A.} (2014a)
\newblock An algorithm for the convolution of {L}egendre series.
\newblock {\em SIAM J. Sci. Comput.}, {\bf 36}, A1207--A1220.

\bibitem[Hale \& Townsend(2014b)Hale \& Townsend]{hale2014a}
{\sc Hale, N. \& Townsend, A.} (2014b)
\newblock A fast, simple, and stable {C}hebyshev--{L}egendre transform using an
  asymptotic formula.
\newblock {\em SIAM J. Sci. Comput.}, {\bf 36}, A148--A167.

\bibitem[Jackiewicz {\em et~al.}(2006)Jackiewicz, Rahman, \&
  Welfert]{jackiewicz2006}
{\sc Jackiewicz, Z., Rahman, M. \& Welfert, B.~D.} (2006)
\newblock Numerical solution of a {F}redholm integro-differential equation
  modelling neural networks.
\newblock {\em Appl. Numer. Math.}, {\bf 56}, 423--432.

\bibitem[Kress(1999)Kress]{kress1999}
{\sc Kress, R.} (1999)
\newblock {\em Linear integral equations\/},  vol.~82, second edn.
\newblock Springer-Verlag, New York, pp. xiv+365.

\bibitem[Kuang(1993)Kuang]{kuang1993}
{\sc Kuang, Y.} (1993)
\newblock {\em Delay differential equations: with applications in population
  dynamics\/},  vol. 191.
\newblock Academic Press.

\bibitem[Lanczos(1956)Lanczos]{lanczos1956}
{\sc Lanczos, C.} (1956)
\newblock {\em Applied analysis\/}.
\newblock Prentice Hall, Inc., Englewood Cliffs, N. J., pp. xx+539.

\bibitem[Linz(1985)Linz]{linz1985}
{\sc Linz, P.} (1985)
\newblock {\em Analytical and numerical methods for {V}olterra equations\/},
  vol.~7.
\newblock SIAM, Philadelphia, PA, pp. xiii+227.

\bibitem[Ma \& Brunner(2006)Ma \& Brunner]{ma2006}
{\sc Ma, J. \& Brunner, H.} (2006)
\newblock A posteriori error estimates of discontinuous {G}alerkin methods for
  non-standard {V}olterra integro-differential equations.
\newblock {\em IMA J. Numer. Anal.}, {\bf 26}, 78--95.

\bibitem[MacCamy(1977)MacCamy]{maccamy1977}
{\sc MacCamy, R.~C.} (1977)
\newblock An integro-differential equation with application in heat flow.
\newblock {\em Quart. Appl. Math.}, {\bf 35}, 1--19.

\bibitem[Mason \& Handscomb(2003)Mason \& Handscomb]{mason2003}
{\sc Mason, J.~C. \& Handscomb, D.~C.} (2003)
\newblock {\em Chebyshev polynomials\/}.
\newblock Chapman \& Hall/CRC, Boca Raton, FL, pp. xiv+341.

\bibitem[Medlock \& Kot(2003)Medlock \& Kot]{medlock2003}
{\sc Medlock, J. \& Kot, M.} (2003)
\newblock Spreading disease: integro-differential equations old and new.
\newblock {\em Math. Biosci.}, {\bf 184}, 201--222.

\bibitem[Olver {\em et~al.}(2017)Olver, Goretkin, Slevinsky, \&
  Townsend]{ApproxFun}
{\sc Olver, S., Goretkin, G., Slevinsky, R.~M. \& Townsend, A.}
\newblock  (2017).
\newblock \url{https://github.com/ApproxFun/ApproxFun.jl}.

\bibitem[Olver \& Townsend(2013)Olver \& Townsend]{olver2013}
{\sc Olver, S. \& Townsend, A.} (2013)
\newblock A fast and well-conditioned spectral method.
\newblock {\em SIAM Rev.}, {\bf 55}, 462--489.

\bibitem[Olver \& Townsend(2014)Olver \& Townsend]{olver2014}
{\sc Olver, S. \& Townsend, A.} (2014)
\newblock A practical framework for infinite-dimensional linear algebra.
\newblock {\em Proceedings of the 1st First Workshop for High Performance
  Technical Computing in Dynamic Languages\/}.
\newblock {\em Proceedings of the 1st First Workshop for High Performance
  Technical Computing in Dynamic Languages\/}., pp. 57--62.

\bibitem[Orszag(1971a)Orszag]{orszag1971b}
{\sc Orszag, S.~A.} (1971a)
\newblock Accurate solution of the {O}rr--{S}ommerfeld stability equation.
\newblock {\em J. Fluid Mech.}, {\bf 50}, 689--703.

\bibitem[Orszag(1971b)Orszag]{orszag1971a}
{\sc Orszag, S.~A.} (1971b)
\newblock Numerical simulation of incompressible flows within simple
  boundaries. {I}. {G}alerkin (spectral) representations.
\newblock {\em Studies in Appl. Math.}, {\bf 50}, 293--327.

\bibitem[Ortiz(1969)Ortiz]{ortiz1969}
{\sc Ortiz, E.~L.} (1969)
\newblock The tau method.
\newblock {\em SIAM J. Numer. Anal.}, {\bf 6}, 480--492.

\bibitem[Slevinsky \& Olver(2017)Slevinsky \& Olver]{slevinsky2016}
{\sc Slevinsky, R.~M. \& Olver, S.} (2017)
\newblock A fast and well-conditioned spectral method for singular integral
  equations.
\newblock {\em J. Comput. Phys.}, {\bf 332}, 290--315.

\bibitem[Tang {\em et~al.}(2008)Tang, Xu, \& Cheng]{tang2008}
{\sc Tang, T., Xu, X. \& Cheng, J.} (2008)
\newblock On spectral methods for {V}olterra integral equations and the
  convergence analysis.
\newblock {\em J. Comput. Math.}, 825--837.

\bibitem[Tankov(2003)Tankov]{tankov2003}
{\sc Tankov, P.} (2003)
\newblock {\em Financial modelling with jump processes\/},  vol.~2.
\newblock CRC press.

\bibitem[Townsend {\em et~al.}(2016)Townsend, Webb, \& Olver]{townsend2016}
{\sc Townsend, A., Webb, M. \& Olver, S.}
\newblock  (2016).
\newblock (submitted).

\bibitem[Townsend \& Trefethen(2013)Townsend \& Trefethen]{chebfun2}
{\sc Townsend, A. \& Trefethen, L.~N.} (2013)
\newblock An extension of {C}hebfun to two dimensions.
\newblock {\em SIAM J. Sci. Comput.}, {\bf 35}, C495--C518.

\bibitem[Trefethen(2013)Trefethen]{trefethen2013}
{\sc Trefethen, L.~N.} (2013)
\newblock {\em Approximation theory and approximation practice\/}.
\newblock SIAM, Philadelphia, PA, pp. viii+305 pp.+back matter.

\bibitem[Trefethen \& Trummer(1987)Trefethen \& Trummer]{trefethen1987}
{\sc Trefethen, L.~N. \& Trummer, M.~R.} (1987)
\newblock An instability phenomenon in spectral methods.
\newblock {\em SIAM J. Numer. Anal.}, {\bf 24}, 1008--1023.

\bibitem[Wang \& Xiang(2012)Wang \& Xiang]{wang2012}
{\sc Wang, H. \& Xiang, S.} (2012)
\newblock On the convergence rates of {L}egendre approximation.
\newblock {\em Math. Comp.}, {\bf 81}, 861--877.

\bibitem[Xiang \& Brunner(2013)Xiang \& Brunner]{xiang2013}
{\sc Xiang, S. \& Brunner, H.} (2013)
\newblock Efficient methods for {V}olterra integral equations with highly
  oscillatory {B}essel kernels.
\newblock {\em BIT\/}, {\bf 53}, 241--263.

\end{thebibliography}

\end{document}